\documentclass{amsart}

\usepackage{amssymb}
\usepackage[stretch=10,shrink=10]{microtype}
\usepackage[showframe=false, margin = 1.5 in]{geometry}
\usepackage{mathtools}
\usepackage[shortlabels]{enumitem}
\usepackage{ifthen}
\usepackage{hyperref}
\usepackage{theoremref,xcolor}

\makeatletter
\def\thmref@flush{%
   \ifx\thmref@last\empty\else
      \ifthmref@comma, \thmref@finaltrue\fi \thmref@commatrue
      \thmref@last \ifx\thmref@stack\empty\else s\fi \thmref@num 0
      \let\do\thmref@one \thmref@stack
      \ifcase\thmref@num\or\space and\else\thmref@finaltrue, and\fi
      ~\ref{\thmref@head}\let\thmref@stack\empty\fi}
\def\thmref@one#1{\ifnum\thmref@num>0,\fi
   \space\ref{#1}\advance\thmref@num 1\relax}
\makeatother

\newcommand{\E}{\mathbf{E}}
\renewcommand{\P}{\mathbf{P}}
\newcommand{\1}{\mathbf{1}}

\DeclareMathOperator{\Ber}{Bernoulli}
\DeclareMathOperator{\Bin}{Bin}

\DeclareMathOperator{\Exp}{Exp}
\DeclareMathOperator{\Unif}{Uniform}
\DeclarePairedDelimiter\abs{\lvert}{\rvert}%
\DeclarePairedDelimiter\ceil{\lceil}{\rceil}%
\newcommand{\bigmid}{\;\big\vert\;}

\newcommand{\eqd}{\overset{d}=}
\newcommand{\RR}{\mathbb{R}}

\newcommand{\Bb}{\mathcal{B}}
\newcommand{\GG}{\mathrm{GG}}

\newcommand{\utailprec}[1]{\preceq_{\overline{#1}}}
\newcommand{\ltailprec}[1]{\preceq_{\underline{#1}}}
\newcommand{\utailsucc}[1]{\succeq_{\overline{#1}}}

\newcommand{\Eq}[2][n]{E^{(#1)}_{#2}}
\newcommand{\tq}[2][none]{\ifthenelse{\equal{#1}{none}}{t_{#2}}{t^{(#1)}_{#2}}}

\newtheorem{thm}{Theorem}
\newtheorem{lemma}[thm]{Lemma}
\newtheorem{prop}[thm]{Proposition}

\theoremstyle{remark}
\newtheorem{remark}[thm]{Remark}
\theoremstyle{definition}
\newtheorem{define}[thm]{Definition}

\makeatletter
\newcommand*\proc{{\mathpalette\bigcdot@{.7}}}
\newcommand*\bigcdot@[2]{\mathbin{\vcenter{\hbox{\scalebox{#2}{$\m@th#1\bullet$}}}}}
\makeatother

\author{Tobias Johnson}
\address{College of Staten Island}
\email{tobias.johnson@csi.cuny.edu}
\author{Erol Pek\"oz}
\address{Boston University}
\email{pekoz@bu.edu}

\title[Concentration inequalities from monotone couplings]{Concentration inequalities from monotone couplings for graphs, walks, trees and branching processes}
 

\begin{document}
  \begin{abstract}
Generalized gamma distributions arise as limits in many settings involving random graphs, walks, trees, and branching processes.  Pek\"oz, R\"ollin, and Ross (2016) exploited characterizing distributional fixed point equations to obtain uniform error bounds for generalized gamma approximations using Stein's method.  Here we show how monotone couplings arising with these fixed point equations can be used to obtain sharper tail bounds that, in many cases, outperform competing moment-based bounds and the uniform bounds obtainable with Stein's method. Applications are given to concentration inequalities for preferential attachment random graphs, branching processes, random walk local time statistics and the size of random subtrees of uniformly random binary rooted plane trees.
  \end{abstract}
  
  \maketitle

  \section{Introduction}
Stein's method is used to obtain distributional approximation error bounds in a wide variety of settings in applied probability where normal, Poisson, gamma and other limits arise.  
The method was introduced by Stein for normal approximation \cite{Stein72}. Chen adapted it
for Poisson approximation \cite{Chen75}, and since then it has been developed in many directions.
Introductions to Stein's method can be found in \cite{Stein1986, BHJ,Chen2010};
we also refer to the surveys \cite{Ross2011, Chatterjee2014, BC2014}. 

One variant of the approach (see \cite{Gold97}, \cite{PRR} and the references therein) starts with a distributional fixed point equation that the limit distribution satisfies and obtains error bounds in terms of how closely both sides of the fixed point equation can be coupled together, i.e., their Wasserstein distance.
The fixed point equation often has a probabilistic interpretation that can be leveraged to achieve a close coupling.  
To illustrate, we define a generalization of the size-bias transform for a random variable.
\begin{define}\label{bbias}
If $X$ is a nonnegative random variable with $0<\E[X^\beta]<\infty,$ we say that $X^{(\beta)}$ is the $\beta-$power bias transform of $X$ if
$$\E[X^\beta]\E[f(X^{(\beta)})]=\E[X^\beta f(X)]$$ holds for all for bounded measurable $f$.
\end{define}

If $X^s \eqd X^{(1)}$, the familiar size-bias transform of $X$,
it can be shown that $X$ satisfies the distributional fixed-point equation $X \eqd X^s -1 $ if and only if $X$  follows a Poisson distribution.
Stein's method can be used to make an approximate version of this statement:
the total variation distance between the law of $X$ and a Poisson distribution
can be bounded by the Wasserstein distance between $X$ and $X^s-1$ \cite[Theorem~4.13]{Ross2011}.
Similarly, the exponential distribution is characterized by the fixed-point equation $X\eqd U X^s$ for an independent $\Unif(0,1)$ variable $U$.
Under the assumption that $X$ has a finite second moment, Stein's method can be used to show that
the Wasserstein distance between $X$ and an exponential distribution is at most twice
the Wasserstein distance between $X$ and $UX^s$ \cite[Theorem~2.1]{Pekoz2011}.

The bounds above allow for accurate approximations of the law of $X$ in the bulk of the distribution.
Tail bounds can also be obtained when the two sides of the distributional fixed-point equation can be coupled
together with some sort of monotonicity. This is carried out for the Poisson distribution
in \cite{Ghosh2011a,AB}. For example, if there exists a coupling of $X^s$ with $X$ so that
$X^s-1\preceq X$ a.s., then $X$ satisfies a Poisson-like tail bound.
\cite{Cook2018} loosened this monotonicity condition to $\P[ X^s-1\leq X \mid X^s\geq x]\geq p$ and applied the results to bounding the second largest eigenvalue of the adjacency matrix of a random regular graph.

The monotonicity condition we study in the current article involves the following definition  from \cite{PRR}.
\begin{define}
For a nonnegative random variable $X$ and $\alpha,\beta>0$,
the variable $X^*$ is the $(\alpha,\beta)$-generalized equilibrium distribution of $X$ if 
  \begin{align*}
    X^*\eqd V_\alpha X^{(\beta)},
  \end{align*}
  where $V_\alpha$ has density $\alpha x^{\alpha-1}\,dx$ on $[0,1]$ and is independent
  of $X^{(\beta)}$. 
  \end{define}

Throughout the paper,  we use the notation $X^*$ to mean a random variable
with the $(\alpha,\beta)$-generalized equilibrium distribution of $X$, with $\alpha$
and $\beta$ specified as necessary.
The familiar equilibrium distribution $X^e$ of $X$ from renewal theory is the case where $\alpha=\beta=1$.
It is well known that the fixed points of the equilibrium transform, i.e., the distributions
satisfying $X^e\eqd X$, are the exponential distributions. Generalizing
this fact, for any $\alpha,\beta>0$ let $\GG(\alpha,\beta)$ denote the $(\alpha,\beta)$-generalized 
gamma distribution,
which has density function $\beta x^{\alpha-1}e^{-x^\beta}/\Gamma(\alpha/\beta)\,dx $ for $x>0$.
The fixed points of the generalized equilibrium transform are the generalized
gamma distributions, up to scaling. That is, for any choice of $\alpha,\beta>0$, the random variable
$X$ satisfies $X^*\eqd X$ if and only if $X\eqd cZ$ where $Z\sim\GG(\alpha,\beta)$
\cite[Theorem~5.1]{KP}.

  The main result of \cite{PRR} is a bound on the Kolmogorov distance between a properly 
  rescaled version of $X$ and $Z\sim\GG(\alpha,\beta)$  when $X$ and $X^*$ are not identical
  but are close in the L\'evy--Prokohorov metric.
  This yields a bound on $\abs{\P[cX>t] - \P[Z>t]}$ that can be used to estimate probabilities
  in the bulk of the distribution of $X$. But since this bound is uniform in $t$,
  it is  too large to be useful for small tail probabilities in  applications.  This is the launching point for the current article.
  
  Let $X\preceq Y$ denote that $X$ is stochastically dominated by $Y$ in the usual stochastic order.
  We introduce two stochastic orders as follows. For a constant $p\in(0,1]$,
  we write $X \utailprec{p} Y$ to denote that $\P[X>t] \leq \P[Y>t]/p$ for all $t\in\RR$.
  Similarly, $X \ltailprec{p} Y$ denotes that $\P[X\leq t]\geq p\P[Y\leq t]$ for all $t\in\RR$.  
  In the $p=1$ case, both orders are the usual one. For $p<1$, they represent two different
  relaxations of the usual order and have alternate characterizations given in \thref{lem:ultail}.

  Our two main result are that stochastic domination of $X^*$ by $X$ in these orders leads to upper and lower tail probability bounds:
  
  \begin{thm}\thlabel{thm:concentration}
    Suppose that $X^*\utailprec{p} X$ for some $p\in(0,1]$ and $\alpha,\beta>0$. 
    Let $\mu=\bigl(\frac{\beta}{\alpha}\E X^\beta\bigr)^{1/\beta}$. If $\alpha\neq\beta$, then
    for any $t\geq 1$,
    \begin{align*}
      \P[X\geq \mu t] \leq t^\alpha e^{2-pt^\beta}.
    \end{align*}
    If $\alpha=\beta$, then for all $t>0$,
    \begin{align*}
      \P[X\geq \mu t] \leq e^{1-pt^\beta}.
    \end{align*}
    
  \end{thm}

  \begin{thm}\thlabel{thm:concentration.lower}
    Suppose $X^*\ltailprec{p} X$ for some $p\in(0,1]$ and $\alpha,\beta>0$. 
    Let $\mu=\bigl(\frac{\beta}{\alpha}\E X^\beta\bigr)^{1/\beta}$. Then
    \begin{align}\label{eq:ltail}
      \P[X\leq \mu t] \leq  
            \biggl(\frac{\alpha}{\beta}\biggr)^{1-\alpha/\beta}\frac{t^\alpha}{\frac{\alpha}{\beta}p- t^\beta} ,
    \end{align}
    so long as $t^\beta<p\alpha/\beta$. If $\alpha=\beta$, an improved bound holds for all $t\geq 0$:
    \begin{align}\label{eq:ltail.alpha=beta}
      \P[X\leq \mu t] \leq \frac{t^\alpha}{p}.
    \end{align}
  \end{thm}

  The factor $\mu$ used to rescale $X$ in these theorems is equal to $1$ when
  $X$ has the $\GG(\alpha,\beta)$ distribution.
  
  We mention that the case $\alpha=\beta=p=1$ case of \thref{thm:concentration} was already proven
  by Mark Brown \cite[Theorem~3.2]{Brown}; see Section~\ref{sec:beta=alpha}.
  
  As an application of these concentration theorems, we prove several results
  for graphs, walks, trees and branching processes.  We next define several quantities from such models and show that the above inequalities apply.
  
{\bf Preferential attachment random graphs}. Consider a preferential attachment random graph model (see \cite{Barabasi1999} and \cite{PRR})
that  starts with an initial ``seed graph'' consisting of one node, or a collection of nodes grouped together, having total weight $w$.
Additional nodes are added sequentially and when a node is added it attaches $l$ edges, one at a time, directed
{from} it {to} either itself or to nodes in the existing graph
according to the following rule: each edge attaches to a potential node with chance proportional to that node's weight right before the moment of attachment,
where incoming edges contribute weight one to a node and each node other than the initial node when added has initial weight one.
The case where $l=1$ is the usual Barabasi-Albert tree with loops but started from a node with initial weight $w$.
Let $W$ be the total weight of the initial ``seed graph'' after an additional~$n$~edges have been added to the graph. 

{\bf Random binary rooted plane trees}. Let $U$ be the number of vertices in the minimal spanning tree spanned by the root and~$k$ randomly chosen distinct leaves of a uniformly chosen binary, rooted plane tree with~$2n-1$ nodes, that is, with~$n$ leaves and~$n-1$ internal nodes.
  
 {\bf Random walk local times}.  Consider the  one-dimensional simple symmetric random walk~$S_n = (S_n(0),\dots,S_n(n))$ of length~$n$ starting at the origin. Define
$$L_n = \sum_{i=0}^{n} 1\{S_n (i) =0\}$$
to be the number of times the random walk visits the origin by time~$n$. Let
$$L^b_{2n}  \sim [L_{2n} \mid S_{2n}(0)= S_{2n}(2n)=0]
$$
be the local time of a random walk bridge. Here we use the notation $[X\mid E]$ to denote
the distribution of a random variable $X$ conditional on an event $E$ with nonzero probability.
 
The next result shows that the above concentration inequalities hold for these quantities.
\begin{thm}\thlabel{thm:graphs.walks.trees}
With the above definitions, \begin{enumerate}[(a)]
\item $W^*\preceq W$    with $\alpha=w,\beta=l+1$
\item
$U^*\preceq U$  with $\alpha=2k,\beta=2$
\item 
$(L_n)^*\preceq L_n$  with $\alpha=1, \beta =2$
\item 
$(L^b_{2n} )^*\preceq L^b_{2n} $  with $\alpha=2, \beta=2$\end{enumerate}
and the conditions and conclusions of Theorem~\ref{thm:concentration} and Theorem~\ref{thm:concentration.lower} hold for $W$, $V$, $L_n$, and $ L^b_{2n}$ with $p=1$ and the corresponding values in (a)-(d) for $\alpha, \beta$.

 \end{thm}

  We also give two tail bounds for Galton--Watson branching processes.
  See Section~\ref{sec:GW.review} for a review of previous bounds.
    \begin{thm}\thlabel{thm:GW1}
    Let $Z_n$ be the size of the $n$th generation of a Galton--Watson process,
    and let $\mu=\E Z_1$, the mean of its child distribution.
    Consider $Z_n^*$ with $\alpha=\beta=1$. If $Z_1^*\preceq Z_1$,
    then $Z_n^*\preceq Z_n$ for all $n$, and
    \begin{align}
      \P[Z_n\geq t\mu^n] &\leq e^{1-t},\label{eq:GW1.1}\\\intertext{and}
      \P[Z_n\leq t\mu^n] &\leq t\label{eq:GW1.2}
    \end{align}
    for all $t>0$.
  \end{thm}

  \thref{thm:GW1} requires the child distribution to be supported on $\{1,2,\ldots\}$, since
  the condition $Z_1^*\preceq Z_1$ fails if $\P[Z_1=0]>0$. The next result relaxes this requirement and allows
  us to consider Galton--Watson trees with a nonzero extinction probability,
  with our concentration result applying conditional on nonextinction.
  We impose a condition on the child distribution that comes from
  reliability theory.
  For a random variable $X$ taking values in the positive integers, we say that $X$
  is \emph{D-IFR} (which stands for \emph{discrete increasing failure rate}) if
  $\P[X=k\mid X\geq k]$ is increasing in $k$ for $k\geq 1$. 
  As we will discuss in Section~\ref{sec:reliability}, if $X$ is D-IFR then
  $X^*\preceq X$ with $\alpha=\beta=1$.
  
  In the following theorem and onward, for a random variable $X$ we use
  $X^>$ to denote a random variable with distribution $[X\mid X>0]$.
  
  \begin{thm}\thlabel{thm:GW2}
    Let $Z_n$ be the size of the $n$th generation of a Galton--Watson tree,
    and suppose that $Z_1^>$ is D-IFR. Then $Z_n^*\preceq Z_n^>$ with $\alpha=\beta=1$, and with
    $m(n)=\E[Z_n\mid Z_n>0]$,
    \begin{align}
      \P[Z_n\geq tm(n)\mid Z_n>0] &\leq e^{1-t},\label{eq:GW2.upper}\\\intertext{and}
      \P[Z_n\leq tm(n)\mid Z_n>0] &\leq t\label{eq:GW2.lower}
    \end{align}
    for all $t>0$.
  \end{thm}

  This theorem holds in subcritical, critical, and supercritical cases alike.
  The difference comes only in the mean $m(n)$, which grows exponentially 
  for a supercritical tree, grows linearly for a critical
  tree, and remains bounded for a subcritical tree.
  
  \thref{thm:graphs.walks.trees,thm:GW1,thm:GW2} apply
  the $p=1$ cases of \thref{thm:concentration,thm:concentration.lower} with the usual stochastic order.
  We use the $p<1$ case in this paper only in \thref{rmk:p<1.use} to sketch a way
  to simplify the proof of \thref{thm:graphs.walks.trees} at the cost of a weaker concentration inequality.
  Nonetheless, we expect that the $p<1$ case will prove useful.
  For the analogous concentration bound in \cite{Cook2018} based on the Poisson distributional
  fixed point $X^s\eqd X+1$, the $p<1$ case is essential to applications
  on random regular graphs (see \cite[Proposition~2.3]{Cook2018} and \cite[Theorem~4.1]{Zhu})
  and on interacting particle systems \cite[Proposition~5.17]{JJLS}.
  
  The tail bounds we produce in this paper are sharp in many circumstances.
  The results of \thref{thm:concentration,thm:concentration.lower} in the $\alpha=\beta$ case
  are sharp, which was known already in the previously proven case $\alpha=\beta=p=1$ for the upper tail
  \cite{BrownSharp}. When $\alpha\neq\beta$, we expect that the factor $t^\alpha$ in the upper tail bound
  is not sharp. We discuss this further in Section~\ref{sec:concentration.sharp}.
  The applications of \thref{thm:concentration,thm:concentration.lower} given
  in \thref{thm:graphs.walks.trees} with $\alpha=\beta$ seem likely to be sharp
  (see \thref{rmk:urn.optimality}), and our results on Galton--Watson processes
  are sharp as well (see Section~\ref{sec:GW.optimality}).
  
  Our main concentration results, \thref{thm:concentration,thm:concentration.lower}, are proven
  in Section~\ref{sec:concentration}.
  In Section~\ref{sec:urns}, we consider an urn model and prove that its counts $N$ satisfy
  $N^*\preceq N$ with $\alpha$ and $\beta$ depending on the model's parameters.
  The random variables $W$, $U$, $L_n$, and $L^b_{2n}$ are expressed in terms of this urn
  model in \cite{PRR}, and \thref{thm:graphs.walks.trees} follows.
  One part of the proof, a regularity property
  for the urn model  similar to log-concavity, is shown by a very technical
  argument in Appendix~A.
  Section~\ref{sec:GW} gives the proofs of \thref{thm:GW1,thm:GW2} on Galton--Watson trees.
  After some background material on reliability theory and on forming the equilibrium transform,
  \thref{thm:GW1} is easy to prove (and in fact is essentially proven in \cite{WDC}).
  The proof of \thref{thm:GW2} is more difficult and requires us to establish some
  delicate properties of the D-IFR class of distributions. In Appendix~B, we
  give some proofs that are well known in reliability theory but
  hard to find in the literature.

  \section{Proof of concentration theorems}\label{sec:concentration}
  Recall that $X\utailprec{p} Y$ and   $X \ltailprec{p} Y$ denote that
  $\P[X>t] \leq \P[Y>t]/p$ and  $\P[X\leq t]\geq p\P[Y\leq t]$, respectively, for all $t\in\RR$.
  We start by characterizing these orders in terms of couplings, along the same lines
  as the standard fact that $X\preceq Y$ if and only if there exists a coupling of $X$
  and $Y$ so that $X\leq Y$ a.s.   We have not seen these stochastic orders defined before, 
  but they are used
  in form~\ref{i:ultail.X>} in \cite{Cook2018} and \cite{DalyJohnson}.
  
  \begin{lemma}\ \thlabel{lem:ultail}
     The following statements are equivalent:
        \begin{enumerate}[(i)]
          \item there exists a coupling $(X,Y)$ for which $\P[X\leq Y\mid X]\geq p$ a.s.;
            \label{i:ultail.X}
          \item there exists a coupling $(X,Y)$ for which $\P[X\leq Y\mid X\geq t] \geq p$ for all
            $t\in\RR$ where the conditional probability is defined;\label{i:ultail.X>}
          \item $X\utailprec{p} Y$;\label{i:ultail.order}
        \end{enumerate}
      as are the following statements:
        \begin{enumerate}[(i)]
          \item there exists a coupling $(X,Y)$ for which $\P[X\leq Y\mid Y]\geq p$ a.s.;
          \item there exists a coupling $(X,Y)$ for which $\P[X\leq Y\mid Y\leq t] \geq p$ for all
            $t\in\RR$ where the conditional probability is defined;
          \item $X\ltailprec{p} Y$.
        \end{enumerate}
  \end{lemma}
  \begin{proof}
    It is clear that \ref{i:ultail.X} implies \ref{i:ultail.X>} in both sets of statements.
    To go from \ref{i:ultail.X>} to \ref{i:ultail.order} in the first set of statements,
    observe that for any $s\in\RR$,
    \begin{align*}
      \P[Y\geq s] &\geq \P[\text{$X\leq Y$ and $X\geq s$}]
         \geq p\P[X\geq s],
    \end{align*}
    applying \ref{i:ultail.X>} in the second inequality.
    Now let $s$ approach $t$ downward to prove \ref{i:ultail.order}.
    A similar argument proves that \ref{i:ultail.X>} implies \ref{i:ultail.order}
    for the second set of statements.
    
    Now we show that \ref{i:ultail.order}$\implies$\ref{i:ultail.X} for the first set of statements.
    Let $B\sim\Ber(p)$ be independent of $X$, and define a random variable $X'$ taking values
    in $[-\infty,\infty)$ by
    \begin{align*}
      X' &= \begin{cases}
        X &\text{if $B=1$,}\\
        -\infty & \text{if $B=0$.}
      \end{cases}
    \end{align*}
    Then
    \begin{align*}
      \P[X'>t] &= p\P[X>t]\leq \P[Y>t]
    \end{align*}
    by\ref{i:ultail.order}.
    Thus $X'$ is stochastically dominated by $Y$ in the standard sense, and therefore
    there exists a coupling
    of $X'$ and $Y$ such that $X'\leq Y$ a.s. 
    (The standard fact that $\P[U>t]\leq P[V>t]$ for all $t$ is equivalent to the existence
    of a coupling for which $U\leq V$ a.s.\ holds when $U$ and $V$ take values in
    $[-\infty,\infty)$, and in fact under considerably more general conditions \cite[Theorem~11]{Strassen}.)
    Under this coupling, given $X$, it holds
    with probability at least $p$ that $X=X'$, in which case $X\leq Y$. Thus $X\utailprec{p} Y$
    implies \ref{i:ultail.X}.
    
    To show that \ref{i:ultail.order}$\implies$\ref{i:ultail.X} for the second set of statements, 
    we use the same idea but define
    $Y'$ to be equal to $Y$ with probability $p$ and equal to $\infty$ with probability $1-p$.
    Then $X\preceq Y'$, yielding a coupling of $X$ and $Y'$ such that $X\leq Y'$ a.s., and
    under this coupling given $Y$ we have $Y=Y'\geq X$ with probability at least $p$.     
  \end{proof}

  Before we get started with our concentration estimates, we make an observation
  that allows us to rescale $\alpha$ and $\beta$.
  \begin{lemma}\thlabel{lem:scaling}
    Let $\alpha$, $\beta$, and $\gamma$ be positive real numbers.
    If $X^*$ has the $(\alpha,\beta)$-generalized equilibrium distribution of $X$, then
    $(X^*)^\gamma$ has the $\bigl(\frac{\alpha}{\gamma},\frac{\beta}{\gamma}\bigr)$-generalized 
    equilibrium distribution of $X^\gamma$.
  \end{lemma}
  \begin{proof}
    For $a>0$, let $V_a$ denote a random variable with density $ax^{a-1}\,dx$, and recall
    that $X^*\eqd V_\alpha X^{(\beta)}$. Now observe that $V_\alpha^\gamma\eqd V_{\alpha/\gamma}$
    and $(X^{(\beta)})^\gamma\eqd(X^\gamma)^{(\beta/\gamma)}$.
  \end{proof}

  \subsection{Upper tail bounds for $\alpha\neq\beta$}
  We start with a technical lemma.
  \begin{lemma}\thlabel{lem:int.bound}
    Suppose that $X^*\utailprec{p} X$ and $\E X^\beta=\alpha/\beta$.
    Let $G(t)=\P[X>t]$.
    For all $t>0$,
    \begin{align}\label{eq:int.bound}
      \int_0^1 u^{\alpha-\beta-1}G(t/u)\,du \leq \frac{G(t)}{p\beta t^\beta}.
    \end{align}
  \end{lemma}
  \begin{proof}
    Since $X^*\utailprec{p} X$,
    \begin{align*}
      G(t)\geq p\P[V_\alpha X^{(\beta)}> t] &= p\int_0^1\alpha u^{\alpha-1}\P[X^{(\beta)}> t/u ] \,du.
    \end{align*}
    By definition of the $\beta$-power bias,
    \begin{align*}
      p\int_0^1\alpha u^{\alpha-1}\P[X^{(\beta)}> t/u ] \,du
        &=p\beta\int_0^1u^{\alpha-1} \E\Bigl[ X^{\beta}\1_{\{X> t/u\}}\Bigr]\,du\\
        &\geq p\beta t^\beta\int_0^1 u^{\alpha-\beta-1}\P[X> t/u]\,du\\&=
        p\beta t^\beta\int_0^1 u^{\alpha-\beta-1}G(t/u)\,du.\qedhere
    \end{align*}
  \end{proof}

  Next, we apply this lemma to deduce bounds on $\E[X\mid X>t]$ in the $\beta-\alpha=1$
  case and on $\E[X^{-1}\mid X>t]$ in the $\beta-\alpha=-1$ case.
  
  \begin{lemma}\thlabel{lem:mrls}
    Suppose that $X^*\utailprec{p} X$ and $\E X^\beta=\alpha/\beta$.
    For all $t\geq 0$ such that $\P[X>t]>0$,
    \begin{align}
      \E\bigl[ X -t\bigmid X>t\bigr] &\leq \frac{1}{p\beta t^{\alpha}}
         &&\text{if $\beta-\alpha=1$}\label{eq:beta>alpha},\\\intertext{and}
       \E\bigl[ X^{-1} -t^{-1}\bigmid X>t\bigr] &\geq 
           - \frac{1}{p\beta t^{\alpha}}
              &&\text{if $\beta-\alpha=-1$}\label{eq:beta<alpha}.
    \end{align}
  \end{lemma}
  \begin{proof}
    Let $G(t)=\P[X>t]$.
    When $\beta-\alpha=1$, we apply \thref{lem:int.bound} and then switch the order of integration to obtain
    \begin{align*}
      \frac{G(t)}{p\beta t^\beta}\geq \int_0^1 u^{-2}G(t/u)\,du &= 
        \E \int_0^1 u^{-2}\1_{\{X>t/u\}}\,du\\
        &= \E\biggl[\1_{\{X>t\}} \int_{t/X}^1 u^{-2}\,du \biggr]\\
        &= \E\Bigl[ \1_{\{X>t\}}\bigl( X/t-1\bigr)
          \Bigr]
        =G(t)\Bigl(t^{-1}\E\bigl[X\bigmid X>t\bigr]
             -1\Bigr).
    \end{align*}
    Canceling the $G(t)$ factors and rearranging terms gives \eqref{eq:beta>alpha}.
    When $\beta-\alpha=-1$, the same approach yields
    \begin{align*}
      \frac{G(t)}{p\beta t^\beta}\geq \int_0^1 G(t/u)\,du 
      = \E\Bigl[ \1_{\{X>t\}} \bigl(1-t/X\bigr)\Bigr] = G(t)\Bigl(1 - t\E\bigl[ X^{-1}\bigmid X>t\bigr]\Bigr),
    \end{align*}
    proving \eqref{eq:beta<alpha}.
  \end{proof}
  
  To give a sense of the purpose of the preceding lemma,  
  the \emph{mean residual life} of a random variable $X$ is the function $m_X(t)=\E[X-t\mid X>t]$,
  with $m_X(t)$ defined as $0$ if $\P[X>t]=0$.
  In general, the distribution of a random variable can be recovered from its
  mean residual life function \cite[Lemma~2]{Meilijson}.
  In \thref{lem:mrls}, we bound the mean residual life
  of $X$ when $\beta-\alpha=1$ or of $-X^{-1}$ when
  $\beta-\alpha=-1$. (A similar approach would bound the mean residual life
  of $\log X$ when $\beta=\alpha$, but a different technique used in Section~\ref{sec:beta=alpha}
  gives better results in that case.) To prove \thref{thm:concentration} when $\alpha\neq\beta$,
  we will first use \thref{lem:scaling} to rescale $\alpha$ and $\beta$ so that $\alpha-\beta=\pm 1$,
  and then we apply the bounds on mean residual life from \thref{lem:mrls} to derive
  tail bounds on $X$.

  \begin{proof}[Proof of \thref{thm:concentration} for $\alpha<\beta$]
    First, we prove the theorem under the assumption that $\mu=1$ (i.e., $\E X^\beta=\alpha/\beta$)
    and that $\beta-\alpha=1$.    
    Let $t_0=(\alpha/p\beta)^{1/\beta}$, and
    let $Z$ be the random variable supported on the interval $(t_0,\infty)$ with
    \begin{align}\label{eq:Zdef}
      \P[Z> t] &=  \biggl(\frac{p\beta e}{\alpha}\biggr)^{\alpha/\beta}t^\alpha e^{-pt^\beta}.
    \end{align}
    We now compute the mean residual life function of $Z$.
    First, by Fubini's theorem, 
    \begin{align}\label{eq:mZ}
      \E\bigl[ (Z-t)\1_{\{Z>t\}}\bigr] = \int_t^{\infty}\P[Z>u]\,du = 
        \frac{1}{p\beta}\biggl(\frac{p\beta e}{\alpha}\biggr)^{\alpha/\beta}e^{-pt^\beta}.
    \end{align}
     Hence, for $t\geq t_0$,
    \begin{align*}
      \E[Z-t\mid Z>t] &= \frac{1}{p\beta t^\alpha }.
    \end{align*}
    By \thref{lem:mrls}, we have $\E[X-t\mid X>t]\leq \E[Z-t\mid Z>t]$ for all
    $t$ where $\P[X>t]>0$.
    By \cite[Theorem~4.A.26]{SS}, we have $\E\varphi(X)\leq\E\varphi(Z)$
    for any increasing convex function $\varphi(x)$.
    
    Now, we apply this statement with the right choice of $\varphi(x)$
    to obtain information about the tail of $X$.
    Fix $t>0$ and define $\varphi(x)=\max\bigl(0,\,(x - t+\gamma)/\gamma\bigr)$ for $\gamma$ to be chosen later
    satisfying $0<\gamma\leq t$.
    Then $\varphi$ is an increasing convex function, and hence $\E\varphi(X)\leq\E\varphi(Z)$.
    Observing that $\1_{\{x\geq t\}}\leq \varphi(x)$,
    \begin{align*}
      \P[X\geq t] \leq \E\varphi(X)&\leq\E\varphi(Z)\\
       &=\gamma^{-1} \E\Bigl[Z - (t-\gamma)\bigr)\1_{\{Z>t-\gamma\}}\Bigr]
       = \frac{1}{\gamma p\beta}\biggl(\frac{p\beta e}{\alpha}\biggr)^{\alpha/\beta}e^{-p(t-\gamma)^\beta}
    \end{align*}
    by \eqref{eq:mZ}, assuming that $t-\gamma\geq t_0$. Since $\beta>1$, the function $x^\beta$
    is convex and its graph lies above its tangent lines. Hence $x^\beta\geq t^\beta + (x-t)\beta t^{\beta-1}$,
    and setting $x=t-\gamma$ we have
    $(t-\gamma)^\beta\geq t^\beta - \gamma\beta t^\alpha$. Using this inequality
    and then minimizing by setting $\gamma=1 / p\beta t^\alpha$ yields
    \begin{align*}
      \P[X\geq t] 
        &\leq \frac{1}{\gamma p\beta}\biggl(\frac{p\beta e}{\alpha}\biggr)^{\alpha/\beta}e^{-p(t^\beta-\gamma\beta t^\alpha)}\\
        &= t^\alpha \biggl(\frac{p\beta e}{\alpha}\biggr)^{\alpha/\beta}
          e^{-pt^\beta+ 1},
    \end{align*}
    so long as $t - 1 / p\beta t^\alpha\geq t_0$.
    To make this bound simpler, first observe that $-x\log x \leq 1-x$ for all $x>0$, since $-x\log x$ is
    concave and $1-x$ is its tangent line at $x=1$. Exponentiating both sides of this inequality,
    setting $x=\alpha/\beta p$, and raising both sides to the $p$th power gives
    \begin{align}\label{eq:pb/a}
      \biggl(\frac{p\beta }{\alpha}\biggr)^{\alpha/\beta}\leq e^{p-\alpha/\beta},
    \end{align}
    from which we obtain
    \begin{align}\label{eq:final.bound}
      \P[X\geq t] &\leq t^\alpha e^{p-pt^\beta +1}\leq t^\alpha e^{2-pt^\beta},
    \end{align}
    still assuming
    \begin{align}
      t-1/p\beta t^\alpha\geq t_0.\label{eq:assumption}
    \end{align}
    
    Now, we show that either \eqref{eq:assumption} is satisfied or the right-hand side of 
    \eqref{eq:final.bound} exceeds $1$. Suppose $t^\alpha e^{2-pt^\beta}<1$. Since we are assuming
    $t\geq 1$, we have $t^\beta\geq 2/p$. Thus our goal is to show that \eqref{eq:assumption} holds
    when $t^\beta\geq 2/p$. Since the left-hand side of \eqref{eq:assumption} is increasing in $t$,
    it suffices to prove that \eqref{eq:assumption} holds when $t=(2/p)^{1/\beta}$.
    To obtain this, start with the inequality $\log(1-x/2)\geq -(\log 2)x$ for $x\in(0,1)$, which holds
    since $\log(1-x/2)$ is concave and equals $-(\log 2)x$ at $x=0,1$.
    Setting $x=1/\beta$ for $\beta>1$ and exponentiating, we obtain
    $1-1/2\beta\geq 2^{-1/\beta}$. Rearranging terms,
    $2^{1/\beta}\bigl(1 - 1/2\beta\bigr)\geq 1$. Hence,
    \begin{align*}
      2^{1/\beta}\biggl(1 - \frac{1}{2\beta}\biggr)\geq \biggl(\frac{\alpha}{\beta}\biggr)^{1/\beta},
    \end{align*}
    since the right-hand side is smaller than $1$.
    Multiplying both sides of this inequality by $p^{-1/\beta}$ and substituting $t=(2/p)^{1/\beta}$,
    \begin{align*}
      t - \frac{t}{2\beta}\geq 
        \biggl(\frac{\alpha}{p\beta}\biggr)^{1/\beta}.
    \end{align*}
    And this is exactly \eqref{eq:assumption}, since $t/2\beta = 1/p\beta t^\alpha$ for
    $t=(2/p)^{1/\beta}$ and $\alpha=\beta-1$.
    This completes the proof assuming $\E X^\beta=\alpha/\beta$
    and $\beta-\alpha=1$.
    
    Now, we drop the assumption $\E X^\beta=\alpha/\beta$
    and relax the assumption $\beta-\alpha=1$ to $\beta-\alpha>0$.
    Let $Y = (X/\mu)^{\beta-\alpha}$ and let $Y^\dagger = (X^*/\mu)^{\beta-\alpha}$,
    and observe that $(X/\mu)^*$ is given by $X^*/\mu$.
    By \thref{lem:scaling},
    the random variable $Y^\dagger$ has the $(\alpha',\beta')$-generalized equilibrium
    distribution of $Y$, where $\alpha'=\alpha/(\beta-\alpha)$ and $\beta'=\beta/(\beta-\alpha)$.
    Now $\beta'-\alpha'=1$ and $\E Y^{\beta'} = \E(X/\mu)^\beta=\alpha/\beta$.
    We now apply the special case of the theorem already proven to obtain
    \begin{align*}
      \P[X\geq \mu t] = \P[Y\geq t^{\beta-\alpha}] &\leq t^{(\beta-\alpha)\alpha'} e^{2-pt^{(\beta-\alpha)\beta'}}
        = t^\alpha e^{2-pt^\beta}
    \end{align*}
    for any $t\geq 1$.
  \end{proof}
  
  The proof when $\alpha>\beta$ follows the same structure, differing only
  in some analytical details.
  \begin{proof}[Proof of \thref{thm:concentration} for $\alpha>\beta$]  
    As in the $\alpha<\beta$ case, it suffices to prove the theorem under the assumptions
    $\mu=1$ (or equivalently $\E X^\beta=\alpha/\beta$)
    and $\alpha-\beta=1$.
    Let $t_0=(\alpha/p\beta)^{\frac{1}{\beta}}$ and
    $Z$ be the random variable supported on $(t_0,\infty)$ defined by \eqref{eq:Zdef}.
    The mean residual life function of $Z$ computed in \eqref{eq:mZ} does not hold
    in the $\alpha-\beta=1$ case, since the computation of the integral in \eqref{eq:mZ}
    relies on $\beta-\alpha=1$, but we can instead observe that
    \begin{align*}
      \P[Z^{-1}\leq t] = \biggl(\frac{p\beta e}{\alpha}\biggr)^{\alpha/\beta} t^{-\alpha}
        e^{ -pt^{-\beta}}
    \end{align*}
    for $0<t<t_0^{-1}$ and then compute
    \begin{align*}
      \E\bigl[(Z^{-1}-t)\1_{\{Z^{-1}<t\}} \bigr] =-\int_0^t \P[Z\leq u]\,du = 
        -\frac{1}{p\beta}\biggl(\frac{p\beta e}{\alpha}\biggr)^{\alpha/\beta}e^{-pt^{-\beta}}.
    \end{align*}
    Hence,
    \begin{align*}
      \E\bigl[Z^{-1}-t\mid Z^{-1}\leq t\bigr] &= -\frac{t^\alpha}{p\beta}.
    \end{align*}
    We then have
    \begin{align*}
      \E\bigl[X^{-1}-t\bigmid X^{-1}\leq t\bigr] \geq \E\bigl[X^{-1}-t\bigmid X^{-1}< t\bigr]
        \geq \E\bigl[Z^{-1}-t\mid Z^{-1}\leq t\bigr]
    \end{align*}
    by \thref{lem:mrls}. 
    By \cite[Theorem~4.A.27]{SS}, it holds for all decreasing convex functions $\varphi(x)$
    that $\E\varphi\bigl(X^{-1}\bigr)\leq\E\varphi\bigl(Z^{-1}\bigr)$.
    Taking $\varphi(x) = \max\bigl(1 + (t-x)/\gamma, 0\bigr)$ for $\gamma>0$
    to be chosen later  and observing
    that $\varphi(x)\geq\1\{x\leq t\}$, we obtain
    \begin{align*}
      \P\bigl[X^{-1}\leq t\bigr] &\leq \E \varphi\bigl(X^{-1}\bigr)\\ &\leq \E\varphi\bigl(Z^{-1}\bigr) 
      = \frac{1}{\gamma}\E\Bigl[
         \bigl(\gamma+t - Z^{-1}\bigr)\1\bigl\{Z^{-1}\leq \gamma+t\bigr\} \Bigr]\\
         &=\frac{1}{p\beta\gamma}\biggl(\frac{p\beta e}{\alpha}\biggr)^{\alpha/\beta}e^{-p(\gamma+t)^{-\beta}},
    \end{align*}
    assuming $\gamma+t<t_0^{-1}$.
    Applying the inequality
    $(\gamma+t)^{-\beta}\geq t^{-\beta}-\beta \gamma t^{-\beta-1}$ which holds by convexity of $x^{-\beta}$
    and then optimizing by setting $\gamma = t^\alpha/p\beta$,
    \begin{align*}
      \P\bigl[X^{-1}\leq t\bigr] &\leq \frac{1}{p\beta\gamma}\biggl(\frac{p\beta e}{\alpha}\biggr)^{\alpha/\beta}e^{-p(t^{-\beta}+\beta\gamma t^{-\alpha})}\\
        &= \biggl(\frac{p\beta e}{\alpha}\biggr)^{\alpha/\beta}t^{-\alpha}e^{1-pt^{-\beta}}.
    \end{align*}
    Thus, under the assumption
    \begin{align}\label{eq:ambass}
      1/p\beta t^{\alpha} + t^{-1} < (p\beta/\alpha)^{1/\beta},
    \end{align}
    we obtain
    \begin{align}\label{eq:finamb1}
      \P[X\geq t]=\P\bigl[X^{-1}\leq t^{-1}\bigr] \leq \biggl(\frac{p\beta e}{\alpha}\biggr)^{\alpha/\beta}t^{\alpha}e^{1-pt^{\beta}}\leq t^\alpha e^{p+1-pt^\beta}\leq t^\alpha e^{2-pt^\beta},
    \end{align}
    applying \eqref{eq:pb/a} in the last step.
    
    Finally, we show that either \eqref{eq:ambass} holds
    or the right-hand side of \eqref{eq:finamb1} exceeds $1$.
    Let $f(t)=2-pt^\beta+(\beta+1)\log t$, so that the right-hand side of \eqref{eq:finamb1} is 
    equal to $e^{f(t)}$. Let $g(t)=1/p\beta t^{\beta+1} + t^{-1}$, the left-hand side of
    \eqref{eq:ambass}. Our goal is to show that for all $t\geq 1$ that
    if $f(t)<0$, then $g(t)<(p\beta/(\beta+1))^{1/\beta}$. It is easy to check that for $t\geq 1$,
    the function $f(t)$ increases and then decreases, with $f(1)>0$. Since $g(t)$ is decreasing,
    it suffices to find a value $s\geq 1$ so that $f(s)\geq 0$ and $g(s)\leq (p\beta/(\beta+1))^{1/\beta}$.
    We claim that this holds for
    \begin{align*}
      s &= \biggl( \frac{2(\beta+1)}{p\beta}\biggr)^{1/\beta}.
    \end{align*}
    To confirm this, we compute
    \begin{align*}
      f(s) &= 2 +\frac{\beta+1}{\beta}\Biggl( \log\biggl(\frac{2(\beta+1)}{p\beta}\biggr)-2\Biggr)
           \geq 2 +\frac{\beta+1}{\beta}\Biggl( \log\biggl(\frac{2(\beta+1)}{\beta}\biggr)-2\Biggr).
    \end{align*}
    A bit of calculus shows that the right-hand side of this inequality is minimized when $(\beta+1)/\beta=e/2$ and is equal to $2-e/2>0$ in this case.
    Finally, we must confirm that $g(s)\leq (p\beta/(\beta+1))^{1/\beta}$. We compute
    \begin{align*}
      g(s) &= \biggl(\frac{p\beta}{2(\beta+1)}\biggr)^{1/\beta}\biggl(\frac{2\beta+3}{2\beta+2}\biggr).
    \end{align*}
    Cancelling a factor of $(\beta/(\beta+1))^{1/\beta}$, we must show that
    \begin{align}\label{eq:g(s)}
      2^{-1/\beta}\biggl(\frac{2\beta+3}{2\beta+2}\biggr)\leq 1.
    \end{align}
    Using the inequality $\log(x+1)-\log x\leq 1/x$,
    \begin{align*}
      \log(2\beta+3)-\log(2\beta+2)\leq\frac{1}{2\beta+2}\leq\frac{1}{2\beta}\leq\frac{\log 2}{\beta},
    \end{align*}
    and exponentiating both sides of this inequality confirms \eqref{eq:g(s)}.
  \end{proof}

  \subsection{Upper tail bounds for $\alpha=\beta$}\label{sec:beta=alpha}
  
  As we mentioned in the introduction, \thref{thm:concentration} in the case $\alpha=\beta=p=1$ was 
  first proven in \cite{Brown}.
  The general $\alpha=\beta$ case with $p=1$ then follows by an application of \thref{lem:scaling},
  and it is not difficult to modify Brown's argument to allow $p<1$.
  To save the reader the effort of
  going back and forth between Brown's paper and this one,
  and to highlight his elegant proof, we present the full argument here.
  
  Let $Z(u)$ be a random variable with the distribution 
  $[X^{(\beta-\alpha)}\mid X^{(\beta-\alpha)}\geq u]$.
  If $U\sim\mu$ is a random variable, let $Z(U)$ denote a random variable whose distribution is the mixture
  governed by $\mu$, i.e.,
  \begin{align*}
    \P\bigl[Z(U)\in B\bigr] = \int \P\bigl[X^{(\beta-\alpha)}\in B\bigmid X^{(\beta-\alpha)}\geq u\bigr] \,\mu(du).
  \end{align*}
  
  \begin{prop}\thlabel{prop:coupling}
    Let $\pi$ be the law of $X$, and define $(U,W)$ as the random variables with joint density
    \begin{align*}
      \1\{0\leq u\leq w\}\frac{\alpha u^{\alpha-1}w^{\beta-\alpha}}{\E X^{\beta}}\,du\,d\pi(w).
    \end{align*}
    Then
    \begin{enumerate}[(a)]
      \item $U\eqd X^*$; \label{i:U.dist}
      \item $W\eqd X^{(\beta)}$; \label{i:W.dist}
      \item for any Borel set $\Bb\subset\RR$,
        \begin{align*}
          \P[U\in\Bb\mid W] = \P[V_\alpha W \in\Bb\mid W] \text{ a.s.,}
        \end{align*}
        where $V_\alpha$ has density $\alpha x^{\alpha-1}$
        on $[0,1]$ and is independent of $W$; \label{i:U|W.dist}
      \item for any Borel set $\Bb\subset\RR$,\label{i:W|U.dist}
        \begin{align*}
          \P[W\in\Bb\mid U] = \P[Z(U)\in \Bb\mid U]\text{ a.s.,}
        \end{align*}
        
      \item $Z(X^*)\eqd X^{(\beta)}$.\label{i:Z(X*)}
    \end{enumerate}
  \end{prop}
  \begin{proof}
    The density of $W$ is $(w^\beta/\E X^\beta)\,d\pi(w)$, proving \ref{i:W.dist}.
    Looking at the conditional density of $U$ given $W=w$, we see that \ref{i:U|W.dist}
    holds. 
    Taking expectations in \ref{i:U|W.dist}, we have $U\eqd V_\alpha W$, and
    together with \ref{i:W.dist} this implies \ref{i:U.dist}.
    Fact~\ref{i:W|U.dist}
    is proven by observing that the conditional density of $W$ given $U=u$ is 
    \begin{align*}
      \1\{w\geq u\}  
      \frac{w^{\beta-\alpha}}{\E [X^{\beta-\alpha}\1\{X\geq u\}]}\,d\pi(w),
    \end{align*}
    which is the density of $Z(u)$, the $(\beta-\alpha)$-power bias transform of $[X\mid X\geq u]$.
    Finally, taking expectations in \ref{i:W|U.dist} gives $W\eqd Z(U)$, and then
    \ref{i:U.dist} and \ref{i:W.dist} prove \ref{i:Z(X*)}.
  \end{proof}

  \begin{lemma}\thlabel{lem:sb.record}
    For any $p\in(0,1]$,
    \begin{align*}
      X^*\utailprec{p} X &\Longrightarrow X^{(\beta)}\utailprec{p} Z(X),\\\intertext{and}
      X^*\ltailprec{p} X &\Longrightarrow X^{(\beta)}\ltailprec{p} Z(X).
    \end{align*}
  \end{lemma}
  \begin{proof}
    First, we observe that $Z(u)$ is stochastically increasing in $u$. Thus we can
    couple the random variables $(Z(u))_{u\geq 0}$ so that $Z(u)\leq Z(v)$ whenever $u\leq v$
    (for example, by coupling all $Z(u)$ to the same $\Unif(0,1)$ random variable
    by the inverse probability transform).
    Under such a coupling, $X^*\leq X$ implies that $Z(X^*)\leq Z(X)$.
    
    Now, suppose $X^*\utailprec{p} X$.
    By \thref{lem:ultail}, there exists
    a coupling $(X,X^*)$ under which $\P[X^*\leq X\mid X^*]\geq p$ a.s.
    Hence, 
    \begin{align*}
      \P[Z(X^*)\leq Z(X)\mid X^*] \geq \P[X^*\leq X\mid X^*]\geq p
      \text{ a.s.}
    \end{align*}
    Thus $\P[Z(X^*)\leq Z(X)\mid Z(X^*)]\geq p$ a.s.\ as well.
    Since $Z(X^*)\eqd X^{(\beta)}$ by \thref{prop:coupling}\ref{i:Z(X*)}, this proves the lemma
    when $X^*\utailprec{p} X$. The proof when $X^*\ltailprec{p} X$ is identical
    except we take conditional expectations given $X$ rather than $X^*$.
  \end{proof}

  The next lemma will be used to prove a tail estimate for $Z(X)$ when $\alpha=\beta$.  
  \begin{lemma}\thlabel{lem:record.exp}
    Let $\mu$ be a probability measure on $[0,\infty)$. Then for any $t>0$,
    \begin{align*}
      \int_{[0,t)} \frac{1}{\mu[x,\infty)}\, \mu(dx) &\leq -\log\mu[t,\infty).
    \end{align*}
  \end{lemma}
  \begin{proof}
    Assume without loss of generality $\mu[t,\infty)>0$.
    For some partition $0=x_0<\cdots<x_n=t$, let $\varphi(x)$ be the step
    function taking value $1 / \mu[x_i,\infty)$ on interval $[x_i,x_{i+1})$
    for $i=0,\ldots,n-1$.
    Then
    \begin{align}
      \int_{[0,t)}\varphi(x)\,\mu(dx) &= \sum_{i=0}^{n-1}\frac{\mu[x_i,x_{i+1})}{\mu[x_i,\infty)}
        \leq \sum_{i=0}^{n-1} -\log\frac{\mu[x_{i+1},\infty)}{\mu[x_i,\infty)} = -\log\mu[t,\infty),
        \label{eq:step.functions}
    \end{align}
    with the inequality holding because $x\leq -\log(1-x)$ for $x\in[0,1)$.
    
    Now, consider any sequence $\varphi_n(x)$ of such step functions where each partition
    refines the last and the mesh size of the partition goes to zero. Then
    $\varphi_n(x)$ converges upward to $1/\mu[x,\infty)$ for all $x\in[0,t)$, and
    \begin{align*}
      \int_{[0,t)}\varphi_n(x)\,\mu(dx)\to\int_{[0,t)}\frac{1}{\mu[x,\infty)}\,\mu(dx)
    \end{align*}
    by the monotone convergence theorem.
    This proves the lemma by \eqref{eq:step.functions}.    
  \end{proof}

  \begin{lemma}\thlabel{lem:record.bound}
    For $\alpha=\beta$ and any random variable $X$,
    \begin{align*}
      \P[Z(X) \geq t] \leq \P[X\geq t]\bigl(1 - \log\P[X\geq t]\bigr).
    \end{align*}
  \end{lemma}
  \begin{proof}
    Let $G(t) = \P[X\geq t]$.
    For any $u\geq 0$,
    \begin{align*}
      \P[Z(u) \geq t] = \1\{u\geq t\} + \1\{u < t\}\frac{G(t)}{G(u)}.
    \end{align*}
    Hence,
    \begin{align*}
      \P[Z(X)\geq t] &= G(t) + G(t)\E\biggl[\frac{\1\{X<t\}}{G(X)}\biggr],
    \end{align*}
    and by \thref{lem:record.exp} we have
    \begin{align*}
      \P[Z(X)\geq t] &\leq G(t) + G(t)\bigl(-\log G(t)\bigr).\qedhere
    \end{align*}
  \end{proof}
  
  \begin{proof}[Proof of \thref{thm:concentration} for $\alpha=\beta$]  
    As in the $\alpha\neq\beta$ cases, it suffices to prove the theorem
    under the assumption $\mu=1$, or equivalently $\E X^\beta=\alpha/\beta=1$.
    From the definition of $X^{(\beta)}$,
    \begin{align*}
      pt^\beta\P[X\geq t] \leq p\E\bigl[X^\beta\1\{X\geq t\}\bigr] = p\P[X^{(\beta)}\geq t].
    \end{align*}
    By \thref{lem:sb.record} followed by \thref{lem:record.bound},
    \begin{align*}
      p\P\bigl[ X^{(\beta)} \geq t \bigr]\leq \P[Z(X)\geq t] \leq 
      \P[X\geq t] \bigl(1 - \log\P[X\geq t]\bigr).
    \end{align*}
    We have now shown that $pt^\beta\P[X\geq t] \leq \P[X\geq t](1-\log \P[X\geq t])$.
    Assuming $\P[X\geq t]>0$, we can divide to obtain $pt^\beta \leq 1-\log\P[X\geq t]$,
      demonstrating that $\P[X\geq t]\leq e^{1-pt^\beta}$.
  \end{proof}

  \subsection{Lower tail bounds}
  \begin{proof}[Proof of \thref{thm:concentration.lower}]
    As in the proof of \thref{thm:concentration}, by rescaling it suffices
    to prove the theorem when $\mu=1$, i.e., $\E X^\beta = \alpha/\beta$.
    From the definition of $\ltailprec{p}$,
    \begin{align*}
      p\P[X\leq t]\leq \P[X^*\leq t]&=\int_0^1 \alpha u^{\alpha-1} \P\bigl[ X^{(\beta)}\leq t/u\bigr]du\\
        &=\frac{\beta}{\alpha} \E\Biggl[\int_0^1 \alpha u^{\alpha-1} X^\beta\1\{X\leq t/u\}\,du\Biggr]\\
        &=\frac{\beta}{\alpha} \E\Biggl[X^\beta\int_0^{\min(t/X,1)}\alpha u^{\alpha-1}\,du\Biggr]\\
        &= \frac{\beta}{\alpha} \E\Bigl[ X^\beta \min\bigl((t/X)^\alpha,\,1\bigr)\Bigr]\\
        &= \frac{\beta}{\alpha} \E\Bigl[ t^\alpha X^{\beta-\alpha}\1\{X>t\} + X^\beta\1\{X\leq t\}\Bigr].
    \end{align*}
    If $\alpha=\beta$, then the expectation is bounded by $t^\alpha$, proving \eqref{eq:ltail.alpha=beta}.
    Otherwise, we bound the two terms in the expectation separately to get
    \begin{align*}
      p\P[X\leq t] &\leq \frac{\beta}{\alpha}\biggl( t^\alpha\E X^{\beta-\alpha} + t^\beta\P[X\leq t]\biggr).
    \end{align*}
    By Jensen's inequality, $\E X^{\beta-\alpha} \leq (\E X^\beta)^{(\beta-\alpha)/\beta}=(\alpha/\beta)^{1-\alpha/\beta}$.
    Applying this bound and rearranging terms yields \eqref{eq:ltail}.
  \end{proof}

  \subsection{Sharpness of bounds}\label{sec:concentration.sharp}
  \thref{thm:concentration,thm:concentration.lower} are nearly optimal when $\alpha=\beta$
  but seem to be missing a factor of $t^{-\beta}$ when $\alpha\neq\beta$.
  First, let us assume that $p=1$.  
  In \cite{BrownSharp}, it is proven that
  \thref{thm:concentration} is sharp when $p=\alpha=\beta=1$, including the constant factor of $e$.
  By \thref{lem:scaling}, the theorem is sharp whenever $\alpha=\beta$.
  For the reader's convenience, we present a family of examples demonstrating that
  \thref{thm:concentration} cannot be improved in this case; it is a discrete counterpart 
  to the example given in \cite{BrownSharp}.
  Choose integers $\mu$ and $n$ and let $p=1/\mu$, and let $X$ have a capped version
  of the geometric distribution with success probability $p$ as follows:
  \begin{align*}
    \P[X=k]=\begin{cases}
      (1-p)^{k-1}p & \text{if $1\leq k \leq n$,}\\
      (1-p)^n & \text{if $k=n+\mu$,}\\
      0 & \text{otherwise.}
    \end{cases}
  \end{align*}
  Then $X^*\preceq X$ with $\alpha=\beta=1$ (easy to check with \thref{prop:reliability}\ref{i:nbue.char}),
  and $\E X = \mu$.
  Now, set $n=(t-1)\mu$ for some integer $t\geq 2$, and we have
  \begin{align*}
    \P[X\geq \mu t] = \P[X\geq n+\mu] = \bigl(1-\mu^{-1})^{(t-1)\mu}.
  \end{align*}
  This converges to $e^{1-t}$ as $m\to\infty$, confirming that there exist examples
  in which $\P[X\geq \mu t]$ comes arbitrarily close to $e^{1-t}$.
  
  When $\alpha\neq\beta$, one would hope for a upper tail bound of $O(t^{\alpha-\beta}e^{-pt^\beta})$
  rather than the $O(t^\alpha e^{-pt^\beta})$ achieved in \thref{thm:concentration}, which
  would match the tail of the generalized gamma distribution.
  But the best tail bound via moments for the generalized gamma distribution loses a factor of
  $t^{-\beta/2}$ (the calculation is similar to the one carried out
  in \thref{rmk:urn.optimality}), and the Chernoff approach of bounding the moment generating
  function used in the proof of \thref{thm:concentration} in the $\alpha\neq\beta$
  case is always inferior to the moment bound \cite{PhilipsNelson}.
  Thus a new approach would be needed for the proof if the optimal
  tail bound is to be achieved.
  Perhaps Brown's proof for the $\alpha=\beta$ case could be adapted when $\alpha\neq\beta$, 
  though we are not sure what the replacement for \thref{lem:record.bound} would be.
    
  As for lower tail bounds, if $X$ has the $(\alpha,\beta)$-generalized gamma distribution, 
  then \thref{thm:concentration.lower}
  applies to $X$ with $\mu=p=1$.
  Up to constants, the $O(t^\alpha)$ bound for $\P[X\leq t]$ shown in \thref{thm:concentration.lower}
  matches the true tail behavior of $X$ as $t\to 0$. 
  
  Now, we show that the dependence on $p$ in the theorem is nearly optimal. Suppose that
  $X^*\preceq X$ for some $\alpha,\beta>0$, and define
  \begin{align*}
    Y = \begin{cases}
          X& \text{with probability~$p$,}\\
          0& \text{with probability~$1-p$,}
    \end{cases}
  \end{align*}
  for some $0<p\leq 1$. Then
  \begin{align*}
    Y^*\eqd X^*\utailprec{p} Y.
  \end{align*}
  Let $\mu_X = \bigl(\frac{\beta}{\alpha}\E X^\beta\bigr)^{1/\beta}$
  and $\mu_Y = \bigl(\frac{\beta}{\alpha}\E Y^\beta\bigr)^{1/\beta} =p^{1/\beta}\mu_X$.
  For any $t\geq 1$, applying the $p=1$ case of \thref{thm:concentration} to $X$,
  \begin{align*}
    \P[Y \geq \mu_Yt] = p\P[X\geq p^{1/\beta}t\mu_X ]\leq
    \begin{cases} 
       p^{1+\alpha/\beta}t^\alpha e^{2-pt^\beta},  &\text{if $\alpha\neq\beta$,}\\
       e^{1-pt^\beta},& \text{if $\alpha=\beta$.}
    \end{cases}
  \end{align*}
  Meanwhile the bound from applying \thref{thm:concentration} to $Y$ is
  \begin{align*}
    \P[Y\geq \mu_Yt] \leq 
    \begin{cases}
      t^\alpha e^{2-pt^{\beta}},&\text{if $\alpha\neq\beta$,}\\
      e^{1-pt^\beta}&\text{if $\alpha=\beta$.}
    \end{cases}
  \end{align*}
  Thus, the bound on $Y$ is equally sharp as the bound on $X$ provided by the
  $p=1$ case of \thref{thm:concentration}, besides losing a factor of $p^{1+\alpha/\beta}$
  when $\alpha\neq\beta$.
  
  For an example showing optimal dependence on $p$ in the lower tail bound, 
  for the sake of simplicity take $\alpha=\beta=1$.
  Choose some $b<1<a$, and define $X$ as the mixture
  \begin{align*}
    X \eqd \begin{cases}
              \Exp(a) & \text{with probability $\frac{a(1-b)}{a-b}$,}\\
              \Exp(b) & \text{with probability $\frac{(a-1)b}{a-b}$,}
            \end{cases}
  \end{align*}
  and observe that $\E X=1$.
  We can compute directly that
  \begin{align}\label{eq:mixture.prob}
    \P[X\leq t] = \frac{1}{a-b} ( a(1-b)(1-e^{-at}) + (a-1)b(1-e^{-bt}))=(a+b-ab)t - O(t^2).
  \end{align}
  
  The equilibrium transform of a mixture is the mixture of the equilibrium transforms,
  with the new mixture governed by the old governor reweighted by expectation (see
  \thref{lem:equib.mixture}). Together with the fact exponential distributions are fixed points
  of the equilibrium transform, this yields
  \begin{align*}
    X^* \eqd \begin{cases}
               \Exp(a) & \text{with probability $\frac{1-b}{a-b}$,}\\
               \Exp(b) & \text{with probability $\frac{a-1}{a-b}$.}
    \end{cases}
  \end{align*}
  With a bit of work, one can show that $X^*\ltailprec{p}X$ with $p = 1/(a+b-ab)$.
  Thus \thref{thm:concentration.lower} yields
  \begin{align*}
    \P[X\leq t] \leq (a+b-ab)t,
  \end{align*}
  matching \eqref{eq:mixture.prob}.
  
  \section{Concentration for urns, graphs, walks, and trees}\label{sec:urns}

  Each of random variables $W$, $U$, $L_n$, and $L^b_{2n}$ in \thref{thm:graphs.walks.trees} can be 
  expressed in terms of an urn model that we describe now.
  An urn starts with black and white balls and draws are made sequentially. After a ball is drawn, 
  it is replaced and another ball of the same color is added to the urn. Also, after every $l$th draw an additional black ball is added to the urn for some $l\geq 1$. As defined in Section 1.2 of \cite{PRR}, let  $P_n^l(b,w)$ denote the distribution of the number of white balls in the urn after $n$ draws have been made when the urn starts with $b \geq 0$ black balls and $w > 0$ white balls, and let $N_n(b,w)\sim P_n^l(b,w)$.

  Let $N_n^{[r]}(b,w)$ be a rising factorial biased version of $N_n(b,w)$, as defined in Lemma~4.2 of
 \cite{PRR}, 
 so that $$\P\bigl[ N_n^{[r]}(b,w) =k\bigr]= \Biggl(c \prod_{i=0}^{r-1} (k+i)\Biggr) \P[N_n(b,w) =k]$$ for some $c$.
 There it is shown that $N_n^{[r]}(b,w) + r \eqd N_n(b,w+r)$.
 We will use this fact to prove concentration for $N_n(1,w)$, 
 but first we relate the rising factorial bias to the power bias $N_n^{(l+1)}(b, w)$.
 \begin{lemma}\thlabel{lem:unfactorialize}
   For all $1\leq l\leq n$ we have
    $ N_{n-l}^{[l+1]}(b,w) + l \succeq N_n^{(l+1)}(b,w).$
 \end{lemma}
 \begin{proof}
   We will show that $\P\bigl[N_{n-l}^{[l+1]}(b,w) + l = k\bigr] / \P\bigl[N_n^{(l+1)}(b,w)=k\bigr]$
   is increasing in $k$ on the union of the support of
   $N_{n-l}^{[l+1]}(b,w) + l$ and $N_n^{(l+1)}(b,w)$, 
   which is $\{w,w+1,\ldots,w+n\}$.
   We have
   \begin{align*}
     \frac{\P\bigl[N_{n-l}^{[l+1]}(b,w) + l = k\bigr]}{\P\bigl[N_n^{(l+1)}(b,w)=k\bigr]}
       &= \frac{C(k-l)\cdots k\P[N_{n-l}(b,w) = k -l]}{k^{l+1} \P[N_n(b,w) = k] },
   \end{align*}
   where $C$ is a value that does not depend on $k$.
   The expression $(k-l)\cdots k/k^{l+1}$ is increasing in $k$, as is
   \begin{align*}
     \frac{\P[N_{n-l}(b,w) = k -l]}{\P[N_n(b,w) = k] } &= 
       \prod_{j=n-l}^{n-1}
       \frac{\P[N_j(b,w) = j-n+k]}{\P[N_{j+1}(b,w) = j-n+k+1] }
   \end{align*}
   for $w\leq k\leq w+n$, since it is $0$ for $w\leq k < w+l$ and
   each factor on the right-hand side in the product is increasing in $k$ by \thref{lem:lc.variant}
   for $w+l\leq k\leq w+n$ (in this range of $k$, the denominators of the fractions in the product
   are all nonzero).
   This proves the desired stochastic domination
   by \cite[Theorem~1.C.1]{SS}.
 \end{proof}
 
 \begin{prop}\label{ineq}
   For all $l,w,n\geq 1$,
   \begin{align*}
     N_n^*(1,w)\preceq N_n(1,w),
   \end{align*}
   where $N_n^*(1,w)$ is the $(w, l+1)$-generalized equilibrium transform of $N_n(1,w)$.
  \end{prop}
 \begin{proof}
 Let $Q_w(n)$ have the distribution of the number of white balls in a regular Polya urn after $n$ draws starting with 1 black and $w$ white balls. 
 For now, assume that $l\leq n$.
We argue that
\begin{align}
N_n(1,w)&\eqd Q_w(N_n(0,w+1)-w-1) \nonumber\\
&\eqd Q_w (N_{n-l}(1,w+1+l)-w-1) \nonumber\\
& \eqd Q_w (N^{[l+1]}_{n-l}(1,w)+l+1-w-1) \nonumber\\
& \succeq \bigl(N^{[l+1]}_{n-l}(1,w) +l \bigr)V_w,\label{eq:urn.fact.bias}
\end{align}
where $V_w$ has density $w x^{w-1}\,dx$ on $[0,1]$ and is independent of $N^{[l+1]}_{n-l}(1,w)$.
The first line is Lemma~4.5 from \cite{PRR}.
In the second line, we use the trivial relation
\begin{align*}
  N_n(0,w+1)\eqd N_{n-l}(1,w+1+l).
\end{align*}
In the third line, we use
 $$N_n(1,w+r)\eqd N_n^{[r]}(1,w)+r,$$which is the statement of Lemma 4.2 of \cite{PRR}.
The final line uses 
\begin{align}\label{eq:Polya}
  Q_w(n)\succeq (n+w) V_w,
\end{align}
which follows from the fact, taken from the proof of Lemma 4.4 of \cite{PRR}, that for independent and identically distributed $\Unif(0,1)$ variables $U_1,U_2,\ldots U_{w-1}$ we can write
$$Q_w(n) \eqd \max_{i=0,1,...,w-1}
(i + \lceil (n + w - i)U_i\rceil)$$
and this implies
$$Q_w(n) \succeq \max_{i=0,1,...,w-1}
 (n + w )U_i =(n+w)V_w.$$
 We now apply \thref{lem:unfactorialize} to \eqref{eq:urn.fact.bias} to obtain
 \begin{align*}
   N_n(1,w)&\succeq V_w N_n^{(l+1)}(1,w)\eqd N_n^*(1,w)
 \end{align*}
 when $l\leq n$.
 
 When $l > n$, the quantity $N_n(1,w)$ is the number of white balls in a regular Polya urn
 after $n$ draws starting with 1 black and $w$ white balls. Using \ref{eq:Polya} and observing
 that the maximum of $N_n(1,w)$ is $n+w$,
 \begin{align*}
   N_n(1,w) &\eqd Q_w(n)\succeq (n+w)V_w\succeq N_n^{(l+1)}(1,w)V_w\eqd N_n^*(1,w).\qedhere
 \end{align*}
 \end{proof}

 \begin{proof}[Proof of \thref{thm:graphs.walks.trees}]
We have $W\sim P_n^l(1,w)$, $U\sim P_{n-k-1}^1(1,2k)$,  and $L_{2n} \sim P_{n}^1(1,1)$ respectively from Remark~1.3, Proposition~2.1, and Proposition~3.4 in \cite{PRR}.
From Proposition~3.2 in \cite{PRR}, we have $L^b_{2n} \sim P_{n}^1(0,1)$,
and $P_n^1(0,1)$ is the same distribution as $P_{n-1}^1(1,2)$.
 The result then follows from Proposition~\ref{ineq} and then noting that the conditions of \thref{thm:concentration,thm:concentration.lower} hold.
 \end{proof}

 \begin{remark}\thlabel{rmk:urn.optimality}
   The rising factorial moments of $N_n(1,w)$ are explicitly computed in \cite[Lemma~4.1]{PRR}.
   When $w=l+1$, the concentration bound given in \thref{thm:graphs.walks.trees}
   is better than the one obtained from these moments via Markov's inequality. For the sake of simplicity,
   we illustrate with the case $l=1$, $w=2$. 
   The result of Part (a) of \thref{thm:graphs.walks.trees} along with \thref{thm:concentration} shows that
   \begin{align}
     \P\bigl[ N_n(1,w)\geq \gamma_n t\bigr] \leq e^{1-t^2},\label{eq:our.bound}
   \end{align}
   where $\gamma_n^2 = \E N_n(1,w)^2$. From \cite[Theorem~1.2]{PRR}, we know that
   $\gamma_n\sim 2\sqrt{n}$.
   
   Now, we compute the concentration inequality given by the rising factorial moments of $N_n(1,w)$.
   Using the notation $x^{[n]} = x(x+1)\cdots(x+n-1)$,
   from \cite[Lemma~4.1]{PRR} we have
   \begin{align*}
     \E \bigl(N_n(1,w)\bigr)^{[2m]} = 2^{[2m]}\prod_{i=0}^{n-1}\biggl(1 + \frac{2m}{2i+3}\biggr) = 
       2^{[2m]}\prod_{i=0}^{m-1}\frac{2n+2i+3}{2i+3}.
   \end{align*}
   The bound given by applying Markov's inequality to this is
   \begin{align}
     \P\bigl[ N_n(1,w)\geq \gamma_n t \bigr] \leq \frac{\E \bigl(N_n(1,w)\bigr)^{[2m]}}{(\gamma_n t)^{[2m]}}
       &= \frac{2^{[2m]}}{(\gamma_n t)^{[2m]}}\prod_{i=0}^{m-1}\frac{2n+2i+3}{2i+3}\nonumber\\
       &= \prod_{i=0}^{m-1}\frac{(2i+2)(2n + 2i + 3)}{(\gamma_nt + 2i)(\gamma_nt + 2i+1)}.
       \label{eq:moment.bound}
   \end{align}
   Let $m^*$ be the minimizing choice of $m$ in \eqref{eq:moment.bound}.
   Some algebra shows that the multiplicand in this expression
   is bounded by $1$ if and only if
   \begin{align*}
     i \leq \frac{\frac{\gamma_n^2}{n}t^2 + \frac{\gamma_n}{n}t - 4 - \frac{6}{n}}
                 {4 + \frac{8}{n} - \frac{4\gamma_n t}{n}}.
   \end{align*}
   Take $t$ to be fixed with respect to $n$.
   From the asymptotics for $\gamma_n$, the right-hand side of this inequality
   converges to $t^2-1$ as $n\to\infty$. Hence either $m^*=\ceil{t^2}-1$ or $m^*=\ceil{t^2}$ 
   when $n$ is sufficiently large.
   The optimal tail bound obtained from the rising factorial moments is then
   \begin{align*}
     \prod_{i=0}^{m^*-1} \frac{(2i+2)\bigl(2+\frac{2i+3}{n}\bigr)}{\Bigl(\frac{\gamma_n}{\sqrt{n}} t +\frac{2i}{\sqrt{n}}\Bigr)\Bigl(\frac{\gamma_n}{\sqrt{n}}t + \frac{2i+1}{\sqrt{n}}\Bigr)},
   \end{align*}
   which converges as $n\to\infty$  to
   \begin{align*}
     \prod_{i=0}^{m^*-1} \frac{(2i+2)(2)}
        {4t^2} = \frac{(m^*)!}{(t^2)^{m^*}}=\Omega(te^{-t^2}),
   \end{align*}
   applying Stirling's approximation to obtain the last estimate.
   Thus this bound is worse than \eqref{eq:our.bound} by a factor of $t$, when $n$ and $t$ are large.
   
   A more involved calculation for the general case $w=\alpha$, $l=\beta-1$ shows that
   the tail bound from moments is on the order of $t^{\alpha-\beta/2}e^{-t^\beta}$.
   Outside of the $\alpha=\beta$ case, our bound is on the order of $t^\alpha e^{-t^\beta}$
   and is outperformed by the moment bound.
 \end{remark}

  \section{Concentration for Galton--Watson processes}
  \label{sec:GW}
  We adopt the terminology from reliability theory that a random variable
  satisfying $X^*\preceq X$ with  $\alpha=\beta=1$ is
  \emph{NBUE}, which stands for ``new better than used in expectation''
  (see \thref{prop:reliability}\ref{i:nbue.char} for the source of this name).
  Since we will often be applying the equilibrium transform to discrete random
  variables (e.g., the child distribution of a Galton--Watson tree),
  we will use the notation $X^e$ to denote the discrete version
  of the $\alpha=\beta=1$ equilibrium tranform, 
  which we can define by setting $X^e=\ceil{X^*}$ with $X^*$ the standard
  $\alpha=\beta=1$ equilibrium transform.
  Equivalently, we can define $X^e$ to be chosen uniformly at random
  from $\{1,2,\ldots,X^s\}$, where $X^s$ is the size-bias transform of $X$.
  Observe that for a random variable $X$ taking values in the nonnegative integers,
  it is a consequence of the coupling interpretation
  of stochastic dominance that
  $X^e\preceq X$ if and only if $X^*\preceq X$.
  
  \subsection{Some concepts from reliability theory}\label{sec:reliability}
  We consider three classes of discrete probability distributions;
  we will state the relationship between the three classes and give some characterizations
  of them. All are standard in the reliability theory literature, sometimes with varying notation.
  
  To define the first class, the \emph{log-concave} distributions, we first define
  a sequence $t_0,t_1,\ldots$ as \emph{log-concave} if
  \begin{enumerate}[(i)]
    \item $t_n^2\geq t_{n-1}t_{n+1}$ for all $n\geq 1$, and
    \item $t_0,t_1,\ldots$ has no internal zeroes (i.e., if $t_i>0$ and $t_k>0$ for some $i<k$, then $t_j>0$ for
      all $i<j<k$).
  \end{enumerate}
  For $X$ taking nonnegative integer values, we say that $X$ is \emph{log-concave} if the sequence
  $\P[X=k]$ for $k\geq 0$ is log-concave.
  
  Next, we recall the class of distributions on the positive integers with the \emph{D-IFR} property,
  which stands for \emph{discrete increasing failure rate}. 
  As we defined in the introduction, the distribution of a
  positive integer--valued random variable $X$ is in this class if
  $\P[X=k\mid X\geq k]$ is increasing for $k\geq 1$, and in that case we say that $X$ is D-IFR.
  Sometimes in the literature, such random variables are just said to be IFR, with it understood
  to use the above definition rather than the continuous version when considering a
  discrete distribution. Sometimes the notation DS-IFR is used to refer to a random
  variable $X$ on the nonnegative integers for which
  $\P[X=k\mid X\geq k]$ is increasing for $k\geq 0$; see for example \cite{PCW}.
  
  As we mentioned in the introduction, a nonnegative random variable $X$ is said to be NBUE if
  $X^*\preceq X$ with $\alpha=\beta=1$. In the reliability theory literature, a random variable
  $X$ taking positive integer values is sometimes said to be D-NBUE if
  \begin{align}\label{eq:dnbue}
    \frac{1}{\E X}\sum_{k=0}^{\infty}\P[X> n+k] \leq \P[X>n]\qquad\text{for all $n\in\{0,1,2,\ldots\}$.}
  \end{align}
  But since the left-hand side of \eqref{eq:dnbue}
  is equal to $\P[X^e>n]$ (see \eqref{eq:Xe>k}), equation~\eqref{eq:dnbue} is equivalent
  to the assertion that $X^e\preceq X$, which holds for $X$ taking positive integer values
  if and only if $X$ is NBUE.
  
  \begin{prop}\thlabel{prop:reliability}
    For a positive integer--valued random variable $X$:
    \begin{enumerate}[(a)]
      \item $X$ is D-IFR if and only if the distributions
        $[X-k\mid X > k]$ are stochastically decreasing for integers $k\geq 0$;
        \label{i:ifr.char}
      \item $X$ is NBUE if and only if
        $\E[X-k\mid X>k]\leq \E X$ for all integers $k\geq 1$; and
        \label{i:nbue.char}
      \item  $\text{$X$ is log-concave} \implies \text{$X$ is D-IFR} \implies \text{$X$ is NBUE}$.
        \label{i:hierarchy}
    \end{enumerate}
  \end{prop}
  These properties are often stated in the reliability theory literature
  (see \cite[Fig.~2]{PCW} and \cite[Lemma~2]{RSZ}).
  Since we have had trouble digging up proofs of some of them, we have provided
  them in Appendix~B.

  Now, we introduce a class of distributions on the nonnegative integers
  that we call NBUEZT, with ZT standing for \emph{zero-truncated}.
  For $X$ taking nonnegative integer values, we say that $X$ is NBUEZT if $X^>$ is NBUE,
  or equivalently if $X^e\preceq X^>$ (recall from the introduction
  that $X^>$ denotes a random variable with the distribution $[X\mid X>0]$.
  
  In the language defined here, \thref{thm:GW1} states that if the child distribution
  of a Galton--Watson process is NBUE, then all generations are NBUE.
  This is a simple consequence of the statement that a random sum of NBUE-many i.i.d.\ NBUE
  summands is NBUE, which was proven in \cite[Corollary~2.2]{WDC} (though we provide
  a more conceptual proof).
  \thref{thm:GW2}
  states that with $L$ the child distribution, if $L^>$ is D-IFR then all generations
  of the process are NBUEZT. This raises a number of questions---for example, if $L$ is only assumed
  to be NBUEZT, then are successive generations NBUEZT?---that we address
  in Section\ref{sec:GW.optimality}.

  \subsection{Forming the equilibrium transform}
  
  First, we give a recipe for forming
  the equilibrium transform of a sum:
  
  \begin{lemma}\thlabel{lem:equib.sum}
    Let $X_1,\ldots,X_n$ be i.i.d.\ nonnegative random variables,
    and let $S=X_1+\cdots+X_n$. Then
    \begin{align}\label{eq:equib.sum.real}
      S^* \eqd \sum_{k=1}^{I-1} X_k + X_I^*,
    \end{align}
    where $I$ is chosen uniformly at random from $\{1,\ldots,n\}$, independent of all else,
    and $X^*$ denotes the $\alpha=\beta=1$ equilibrium transform of $X$.
    If $X_1,\ldots,X_n$ are integer-valued, then
    \begin{align}\label{eq:equib.sum}
      S^e \eqd\sum_{k=1}^{I-1} X_k + X_I^e.
    \end{align}
  \end{lemma}
  \begin{proof}
    Equation~\ref{eq:equib.sum.real} is the special case of \cite[Theorem~4.1]{PR_nearly_critical} 
    in which $X_1,\ldots,X_n$ are i.i.d.
    When $X_1,\ldots,X_n$ are integer-valued, applying the ceiling function to both sides of 
    \eqref{eq:equib.sum.real} gives \eqref{eq:equib.sum}.
  \end{proof}

  Next, we consider the equilibrium transform of a mixture.
  To give notation for a mixture, let $h$ be a probability measure on the real numbers.
  Suppose that for each $b$ in the support of $h$,
  we have a random variable $X_b$ with distribution $\nu_b$ and mean $m_b\in [0,\infty)$.
  Also assume that $b\mapsto\mu_b$ is measurable.
  The random variable $X$ is the \emph{mixture of $(X_b)$ governed by $h$} if for all bounded
  measurable functions $g$,
  \begin{align*}
    \E g(X) = \int \E g(X_b)\,dh(b).
  \end{align*}
  The basic recipe for the equilibrium transform $X^e$ is that it is a mixture
  of the equilibrium transforms $X_b^e$, governed by a biased version of $h$.
  The analogous recipe works for forming the size-bias transform of a mixture,
  and this result follows from that.
  
  \begin{lemma}\thlabel{lem:equib.mixture}
    Let $X$ be the mixture of $(X_b)$ governed by $h$ as described above,
    and assume that $\E X <\infty$.
    Define the measure $h^s$ by its Radon--Nikodym derivative:
    \begin{align*}
      \frac{d h^s(b)}{dh} = \frac{m(b)}{\E X}.
    \end{align*}
    Then the distribution of $X^e$ is the mixture of $\bigl(X_b^e\bigr)$
    governed by $h^s$.
  \end{lemma}
  \begin{proof}
    By \cite[Lemma~2.4]{AGK}, the size-bias transform $X^s$
    is distributed as the mixture of $X^s_b$ governed by $h^s$.
    With $U\sim\Unif(0,1)$ independent of all else,
    the equilibrium transform $\ceil{UX^s}$ is thus the mixture
    of $\ceil{UX^s_b}$ governed by $h^s$.
  \end{proof}

  \subsection{Proofs of the concentration theorems for Galton--Watson trees}

  \thref{thm:GW1} is a simple consequence of the following statement that an NBUE
  quantity of i.i.d.\ NBUE summands is NBUE. This fact was previously
  proven in \cite[Corollary~2.2]{PCW}. We include our proof here, as it
  takes a very different
  approach from theirs.
  \begin{prop}\thlabel{prop:NBUE.sum}
    Let $X,\,X_1,\,X_2,\,\ldots$ be i.i.d.\ nonnegative random variables, 
    and let $L$ be a positive integer--valued random variable
    independent of all else. Suppose that $X$ and $L$ are NBUE.
    Then the random sum $S = \sum_{k=1}^L X_k$ is NBUE as well.
  \end{prop}
  \begin{proof}
     We construct $S^*$ using \thref{lem:equib.sum,lem:equib.mixture}, as was done
     in \cite[Theorem~3.1]{Pekoz2011}.
     Let $S_n=\sum_{k=1}^n X_k$, so that $S=S_L$, and
     let $T_k\eqd S_k^*$.
     By \thref{lem:equib.sum},
     \begin{align*}
       T_k \eqd X_1+\cdots+X_{I_k-1} + X_{I_k}^*,
    \end{align*}
    where $I_k$ is chosen uniformly at random from $\{1,\ldots,k\}$.
    
    By \thref{lem:equib.mixture}, the equilibrium transform of $S$
    is a mixture of $T_k$ governed by a distribution whose
    Radon--Nikodym derivative with respect to $L$ is
    \begin{align*}
      \frac{\E S_k}{\E S} = \frac{k\E X}{\E L\, \E X}=\frac{k}{\E L},\qquad k\in\mathbb{N},
    \end{align*}
    which is exactly the Radon--Nikodym distribution of $L^s$ with respect to $L$.
     Hence,
    \begin{align*}
      S^* \eqd T_{L^s} &= X_1+\cdots+X_{I_{L^s}-1} + X_{I_{L^s}}^*,\\
        &\preceq X_1 +\cdots+X_{I_{L^s}-1} + X_{I_{L^s}},
    \end{align*}
    with the second line following because $X$ is NBUE.
    Since $I_{L^s}$ is a uniform selection from $\{1,\ldots,L^s\}$, it is the discrete equilibrium
    transform of $L$. Hence $I_{L^s}\preceq L$, and
    \begin{align*}
      S^*&\preceq X_1+\cdots+X_L=S.\qedhere
    \end{align*}
  \end{proof}

  \begin{proof}[Proof of \thref{thm:GW1}]
    Let $L$ be a random variable whose distribution is the child distribution of the tree. 
    Each generation of the Galton--Watson process is the sum
    of $L$ independent copies of the previous generation, i.e.,
    \begin{align*}
      Z_{n+1} &\eqd \sum_{j=1}^L Z_n^{(j)},
    \end{align*}
    where $Z_n^{(j)}$ for $j\geq 1$ denote independent copies of $Z_n$ and $L$ is indepedent
    of $(Z_n^{(j)},\,j\geq 1)$.
    Observing that $Z_1^{(j)}\eqd L$ is NBUE, we can apply \thref{prop:NBUE.sum}
    inductively to conclude that $Z_n$ is NBUE for all $n$.
    The concentration inequalities \eqref{eq:GW1.1} and \eqref{eq:GW1.2} then follow 
    from Theorems~\ref{thm:concentration} and \ref{thm:concentration.lower} with $\alpha=\beta=1$.
  \end{proof}

  For \thref{thm:GW2}, it would be nice
  to argue that an NBUEZT quantity of NBUEZT summands remain NBUEZT, but we have not been able
  to prove or disprove this (see Section~\ref{sec:GW.optimality}).
  But we can show the following weaker statement. For a random variable $X$ taking nonnegative
  integer values, we write $\Bin(X,p)$ to denote the distribution obtained by thinning $X$ by $p$
  (i.e., the sum of $X$ independent $\Ber(p)$ random variables).
  
  \begin{prop}\thlabel{lem:random.sums}
    Let $X_1,X_2,\ldots$ be i.i.d.\ and NBUEZT. Let $p=\P[X_i\geq 1]$, and
    let $L$ be a random variable taking nonnegative integer values, independent of $X_1,X_2,\ldots$,
    such that $\Bin(L,p)$ is NBUEZT. Then $X_1+\cdots+X_L$ is NBUEZT.
  \end{prop}
  \begin{proof}
    Let $S=X_1+\cdots+X_L$, and let $M$ be the number of the random variables $X_1,\ldots,X_L$
    that are nonzero. Then
    \begin{align}\label{eq:Srep}
      S\eqd \sum_{k=1}^M X_k^>,
    \end{align}
    where $M\sim\Bin(L,p)$ is independent of $(X_i^>)_{i\geq 1}$.
    Since $S$ is then a sum of $M$ many strictly positive random variables, it is
    positive if and only if $M$ is positive. Hence
    \begin{align*}
      S^>\eqd \sum_{k=1}^{M^>} X_k^>.
    \end{align*}
    Since $M$ and $X_k$ are NBUEZT, their conditioned versions $M^>$ and $X_k^>$
    are NBUE. Thus $S^>$ is NBUE by \thref{prop:NBUE.sum}, and hence $S$ is NBUEZT.
  \end{proof}
  
  To apply \thref{lem:random.sums}, the NBUEZT property for $L$ must be preserved under thinning.
  We now show that this holds when $L^>$ is D-IFR.
  
  \begin{lemma}\thlabel{lem:IFR.thinning}
    Let $L$ be a random variable taking nonnegative integer values.
    If $L^>$ is D-IFR, then $\Bin(L,p)^>$ is D-IFR for all $0< p\leq 1$.
  \end{lemma}
  \begin{proof}
    Let $(B_k)_{k\geq 1}$ be i.i.d.-$\Ber(p)$ for arbitrary $p\in(0,1)$, and let
    $M=B_1+\cdots+B_L$, so that $M\sim\Bin(L,p)$.
    Our goal is to show that $\P[M=n\mid M\geq n]$ is increasing for $n\geq 1$.
    
    Define
    \begin{align*}
      \varphi(t) = \E\bigl[ (1-p)^{L-t} \bigmid L\geq t \bigr],
    \end{align*}
    which is the conditional probability that $B_{t+1}=\cdots=B_L=0$ given that $L\geq t$.
    Let $T_n$ be the smallest index $t$ such that $B_1+\cdots+B_t=n$.
    We make the following claims:
    \begin{enumerate}[(i)]
      \item $\P[M=n\mid M\geq n] = \E\bigl[\varphi(T_n) \bigmid L\geq T_n \bigr]$;
        \label{i:s1}
      \item the function $\varphi(t)$ is increasing for integers $t\geq 1$;
        \label{i:s2}
      \item the distributions $[T_n \mid L\geq T_n]$ are stochastically increasing in $n$.
        \label{i:s3}
    \end{enumerate}
    To prove \ref{i:s1}, we start by observing that $M=n$ holds if and only $L\geq T_n$
    and $B_{T_n+1},\ldots,B_L$ are all zero. Thus,
    \begin{align*}
      \P[M = n] &= \P[\text{$B_{T_n+1}=\cdots=B_L=0$ and $L\geq T_n$}]\\
                &= \sum_{t,\ell}\P[T_n=t,\,L=\ell,\,B_{t+1}=\cdots=B_\ell=0]\1\{\ell\geq t\}\\
                &= \sum_{t,\ell}\P[T_n=t,\,L=\ell](1-p)^{\ell-t}\1\{\ell\geq t\}\\
                &= \sum_t \P[T_n=t]\sum_\ell \P[L=\ell](1-p)^{\ell-t}\1\{\ell\geq t\}\\
                &= \sum_t \P[T_n=t]\,\E\bigl[ (1-p)^{L-t}\1\{L\geq t\}\bigr]
                = \sum_t \P[T_n=t]\varphi(t)\P[L\geq t].
    \end{align*}
    Since $L$ and $T_n$ are independent, we obtain
    \begin{align*}
      \P[M = n] &= \sum_t \P[T_n=t,\,L\geq t]\varphi(t)\\
        &= \sum_{t,\ell}\P[T_n=t,\,L=\ell]\varphi(t)\1\{\ell\geq t\}
        = \E\bigl[\varphi(T_n)\1\{L\geq T_n\}\bigr].
    \end{align*}
    Observing that the events $\{M\geq n\}$ and $\{L\geq T_n\}$ are the same and dividing
    both sides of the above equation by its probability yields \ref{i:s1}.
    
    Now we prove \ref{i:s2}. Since $L$ is D-IFR, the distributions $[L-t\mid L\geq t]$
    are stochastically decreasing by \thref{prop:reliability}\ref{i:ifr.char}. Thus $\varphi(t)$ is obtained
    by taking the expectation of the decreasing function $x\mapsto (1-p)^x$ under a stochastically
    decreasing sequence of distributions, showing that $\varphi(t)$ is increasing in $t$.
    
    To prove \ref{i:s3}, it suffices (see \cite[Theorem~1.C.1]{SS}) to show that
    \begin{align}\label{eq:s3}
      \frac{\P[T_{n+1}=k\mid L\geq T_{n+1}]}{\P[T_n=k\mid L\geq T_n]}
      \quad\quad\text{is increasing in $k$ for $k\geq n$.}
    \end{align}
    Thus we consider
    \begin{align*}
      \frac{\P[T_{n+1}=k\mid L\geq T_{n+1}]}{\P[T_n=k\mid L\geq T_n]}
        &= \frac{\P[T_{n+1}=k,\,L\geq k]}{\P[L\geq T_{n+1}]}
        \cdot \frac{\P[L\geq T_n]}{\P[T_n=k,\,L\geq k]}\\
        &= \frac{\P[T_{n+1}=k]}{\P[T_n=k]}\cdot \frac{\P[L\geq T_n]}{\P[L\geq T_{n+1}]},
    \end{align*}
    with the second line following from the independence of $T_n$ and $T_{n+1}$ from $L$.
    The final bit is to compute probabilities for the negative binomial distribution:
    \begin{align*}
      \frac{\P[T_{n+1}=k]}{\P[T_n=k]} &= 
        \frac{\binom{k-1}{n}p^{n+1}(1-p)^{k-n-1}}{\binom{k-1}{n-1}p^n(1-p)^{k-n}} = \frac{(k-n)p}{n(1-p)}.
    \end{align*}
    This is increasing in $k$ for $k\geq n$, which proves \eqref{eq:s3}.
    
    Now, statements \ref{i:s1}--\ref{i:s3} combine to prove the lemma:
     the quantity $\E[\varphi(T_n)\mid  L\geq T_n]$
     is increasing in $n$ by \ref{i:s2} and \ref{i:s3}, and hence $M^>$ is D-IFR by \ref{i:s1}.
  \end{proof}

  \begin{proof}[Proof of \thref{thm:GW2}]
    Let $L$ be the child distribution of the tree.
    By \thref{prop:reliability}\ref{i:hierarchy} and \thref{lem:IFR.thinning},
    all thinnings of $L$ are NBUEZT. Hence \thref{lem:random.sums}
    applies and shows that
    \begin{align*}
      \sum_{k=1}^L X_k
    \end{align*}
    is NBUEZT whenever $(X_k)_k\geq 1$ are an i.i.d.\ family of NBUEZT random variables.
    Applying this inductively to each generation $Z_n$ of the Galton--Watson process
    shows that $Z_n$ is NBUEZT for all $n$.
    Therefore \thref{thm:concentration,thm:concentration.lower} apply to $Z_n^>$
    and prove \eqref{eq:GW2.upper} and \eqref{eq:GW2.lower}.
  \end{proof}

  \subsection{On the sharpness and optimality of these results}
  \label{sec:GW.optimality}
  Consider a Galton--Watson process whose child distribution is geometric
  with success probability $p$ on $\{1,2,\ldots\}$, which is NBUE.
  The size of the $n$th generation is
  geometric with success probability $p^n$. Then $Z_n/\mu^n\to\Exp(1)$ in law as $n\to\infty$.
  Hence both upper and lower bounds in \thref{thm:GW1}
  are sharp, besides the extra factor of $e$ in the upper bound.

  If $Z_n$ is the $n$th generation of a critical Galton--Watson tree with $m(n)=\E [Z_n\mid Z_n>0]$,
  then $Z_n/m(n)\to\Exp(1)$ in law as $n\to\infty$ \cite[Theorem~I.9.2]{AthreyaNey}.
  Thus \thref{thm:GW2} provides optimal bounds in this case, again besides the extra factor of $e$
  in the upper bound.
  
  We can also ask whether the conditions of \thref{thm:GW1,thm:GW2} 
  could be weakened and more broadly what properties of the child distribution are
  preserved for all generations.
  \thref{thm:GW1} states that if the child distribution of a Galton--Watson
  process is NBUE, then all its generations are NBUE.
  Log-concave and D-IFR distributions are not preserved in this way.
  For a counterexample, consider a child distribution placing probability $1/8$ on $1$,
  probability $49/64$ on $2$, and probability $7/64$ on $3$. This distribution is
  log-concave and D-IFR, but the size of the second generation is neither.
    
  \thref{lem:IFR.thinning} states that if $L^>$ is D-IFR, then $\Bin(L,p)^>$ is D-IFR.
    The NBUEZT property is not preserved under thinning in this way. Let
    \begin{align*}
      L=\begin{cases}
                1 & \text{with probability $89/100$,}\\
                2 & \text{with probability $109/1000$,}\\
                3 & \text{with probability $9/10000$,}\\
                4 & \text{with probability $1/11250$,}\\
                5 & \text{with probability $1/90000$.}
      \end{cases}
    \end{align*}
    We leave it as an exercise that this distribution
    is NBUEZT (in fact, NBUE) but that all of its thinnings fail to be NBUEZT.
    
    It seems more natural (and would be a weaker condition) to assume only
    that the child distribution $L$ is NBUEZT in \thref{thm:GW2}, rather
    than that $L^>$ is D-IFR. But since    
    the NBUEZT property is not preserved by thinning, our proof of
    \thref{thm:GW2} does not go through with this change. 
    In fact, we are truly unsure the theorem holds with the weaker
    condition.

  \subsection{Previous concentration results for Galton--Watson processes}
  \label{sec:GW.review}  
  Let $Z_n$ be the $n$th generation of a Galton--Watson process whose child distribution has mean $\mu>1$.
  Let $W$ be the almost sure limit of $Z_n/\mu^n$, which exists and is nondegenerate
  when $\E[Z_1\log Z_1]<\infty$. 
  In \thref{thm:GW1,thm:GW2}, properties of the child distribution continue to hold
  for $Z_n/\mu^n$ at all generations. This is in a similar spirit to many results linking
  properties of the child distribution to those of $W$. For example, for $\alpha>1$ it holds
  that $\E Z_1^\alpha$ is finite if and only if $\E W^\alpha$ is finite \cite{BD74}.
  Similarly, $Z_1$ has a regularly varying distribution with index $\alpha>1$ if and only if $W$ does
  \cite{deMeyer}.

  One line of results is on the right tail when the child distribution is bounded.
  Let $d$ be its maximum value, and let $\gamma=\log d/\log\mu>1$. Biggins and Bingham \cite{BB91} 
  used a classic result of Harris \cite{Harris48} to show that
  \begin{align}\label{eq:BB.upper}
    -\log\P[W\geq x] = x^{\gamma/(\gamma-1)}N(x) + o\bigl(x^{\gamma/(\gamma-1)}\bigr),
  \end{align}
  where $N(x)$ is a continuous, multiplicatively periodic
  function. Hence, in the limit the tail of $Z_n/\mu_n$ decays faster than exponentially.
  Fleischmann and Wachtel give a more precise version of this result \cite[Remark~3]{FW2}, showing
  that the tail of $W$ decays as
  \begin{align}\label{eq:FW.upper}
    N_2(x)x^{-\gamma/2(\gamma-1)}\exp\Bigl(-x^{\gamma/(\gamma-1)}N(x)\Bigr),
  \end{align}
  where $N(x)$ and $N_2(x)$ are continuous, multiplicatively periodic function.
  Biggins and Bingham give a version of their result that applies directly to $Z_n$ rather than its
  limit, and more detailed results on the right tail of $Z_n$ in this situation can also be obtained
  from combinatorial results of Flajolet and Odlyzko \cite[Theorem 1]{FO}.
  
  Results on the right tail are also available when the child distribution is heavy-tailed.
  When $Z_1$ satisfies
  \begin{align*}
    \sup_x \frac{\P[Z_1>x/2]}{\P[Z_1>x]} < \infty,
  \end{align*}
  the tails of $Z_n$ satisfy
  \begin{align*}
    c_1\P[Z_1>x] \leq \P[Z_n>\mu^nx] \leq c_2\P[Z_1>x]
  \end{align*}
  for constants $c_1>0$ and $c_2<\infty$ independent of $x$ and $n$ \cite[Theorem~1]{DKW}.
  This result applies, for instance, when the tail of $Z_1$ has polynomial decay.
  A similar result \cite[Theorem~3]{DKW} holds when the tail of $Z_1$ behaves
  like $e^{-x^\alpha}$ for $0<\alpha<1$.
  
  For the left tail of $Z_n$, the behavior depends on the weight
  that the child distribution places on $0$ and $1$. It is known as the
  \emph{Schr\"oder case} when positive weight is placed on those values
  and as the \emph{B\"ottcher case} when it is not. Roughly speaking, the left tail 
  in the Schr\"oder case behaves similarly to the right tail in the heavy-tailed case,
  while the left tail in the B\"ottcher case behaves similarly to the right tail in
  the bounded child distribution case.
  For example, suppose that the child distribution places no weight on $0$ and weight $p_1$ on $1$.
  In the Schr\"oder case, where $p_1> 0$, let $\alpha=-\log p_1 /\log \mu$.
  Then $\P[W\leq x]$ behaves like $x^\alpha$ as $x\to 0$ \cite{DubucLoi}.
  Note that $\alpha=1$ for a geometric child distribution, coinciding with the lower tail
  bound we prove in \thref{thm:GW1}.
  For the B\"ottcher case, where $p_1=0$, let $d\geq 2$ be the minimum value taken by child distribution,
  and let $\beta = \log d/\log\mu\in(0,1)$. Then $-\log\P[W\leq x]$ behaves like
  $x^{-\beta/(1-\beta)}$. A result like \eqref{eq:BB.upper} is shown
  in \cite[Theorem~3]{BB91}, and finer asymptotics along the lines of \eqref{eq:FW.upper}
  are given in \cite[Theorem~1]{FW2}.
  
  Our results apply best to distributions that are unbounded
  but have exponential tails, a case that seems poorly covered
  by the existing literature. Our bound is also more explicit than any we have encountered,
  with no limits or unspecified constants.
    
\section*{Appendix A}
Recall that $P_n^l(b,w)$ is the distribution of the number of white balls after $n$ draws
in the urn model defined in Section~\ref{sec:urns}, and let $N_n(b,w)\sim P_n^l(b,w)$.
This appendix is dedicated to proving the following result, which is used
in \thref{lem:unfactorialize} to compare the rising factorial bias transform of these distributions
to their power bias transforms.
 The property proven in the following lemma is something like log-concavity of the sequence
 $\P[N_n(b,w)=k]$ for fixed $n$
 (which is proven along the way, in \thref{lem:log.concave}), but it involves varying both $k$ and $n$.
 We cannot give much intuition for the proof; it seems to us to be a technical fact that happens
 to be true and can be proven by pushing symbols around in the right way.
 \begin{lemma}\thlabel{lem:lc.variant}
   For all $n,k\geq 0$,
   \begin{align}\label{eq:lc.variant}
     \P[N_n(b,w)=k-1]\,\P[N_{n+1}(b,w)=k+1] &\leq \P[N_n(b,w)=k]\,\P[N_{n+1}(b,w)=k].
   \end{align}
 \end{lemma}

We will in fact prove \thref{lem:lc.variant} for a slightly
generalized version
of the urn process.
As with that process,
start with $b\geq 1$ black balls and $w\geq 1$ white balls, and after each draw
add an extra ball with the same color as the ball drawn.
Instead of adding an additional black ball after every $l$th draw,
we allow black balls to be added at arbitrary but predetermined times.
Thus the number of balls in the urn after $n$ draws, denoted by $B_n$, is an arbitrary but deterministic
strictly increasing sequence with
$B_0=b+w$. Let $N_n$ be the number of white balls in the urn after $n$ draws.
Let $\tq[n]{k}=\P[N_n=k]$.
The dynamics of the urn process gives
\begin{align}\label{eq:recurrence}
  \tq[n+1]{k} &= \biggl(\frac{k-1}{B_n}\biggr) \tq[n]{k-1} + \biggl(1 - \frac{k}{B_n}\biggr)\tq[n]{k}.
\end{align}

First, we show that $\tq[n]{k}$ is log-concave in $k$ for each fixed $n$:
\begin{lemma}\thlabel{lem:log.concave}
  For all $n\geq 0$ and all $k$,
  \begin{align}
    \bigl(\tq[n]{k}\bigr)^2 \geq \tq[n]{k-1}\tq[n]{k+1}.\label{eq:log.concave}
  \end{align}
\end{lemma}
\begin{proof}
  We prove this by induction. For the base case, we have $\tq[0]{k} = \1\{k=w\}$,
  and hence the right-hand side of \eqref{eq:log.concave} is always zero when $n=0$.
  Now, we expand $\bigl(\tq[n+1]{k}\bigr)^2 - \tq[n+1]{k-1}\tq[n+1]{k+1}$ using
  \eqref{eq:recurrence} as $(A_1+A_2+A_3)/B_n^2$ for 
  \begin{align}
     A_1 &= (k-1)^2\tq{k-1}^2 - (k-2)k\tq{k-2}\tq{k},\label{eq:A1}\\
     A_2 &= (B_n-k)^2\tq{k}^2 - (B_n-k+1)(B_n-k-1)\tq{k-1}\tq{k+1},\label{eq:A2}\\
     A_3 &= 2(k-1)(B_n-k)\tq{k-1}\tq{k} - (k-2)(B_n-k-1)\tq{k-2}\tq{k+1} - k(B_n-k+1)\tq{k-1}\tq{k}\label{eq:A3},
  \end{align}
  where we have simplified notation by writing $\tq{k}$ for $\tq[n]{k}$.
  Applying the inductive hypothesis to \eqref{eq:A1} and \eqref{eq:A2} gives
  \begin{align*}
    A_1 &\geq \Bigl((k-1)^2-(k-2)k\Bigr)\tq{k-1}^2 = \tq{k-1}^2,\\\intertext{and}
    A_2 &\geq \Bigl((B_n-k)^2-(B_n-k+1)(B_n-k-1)\Bigr)\tq{k}^2 = \tq{k}^2.
  \end{align*}
  To bound $A_3$, we note that the inductive hypothesis implies $\tq{k-2}\tq{k+1}\leq \tq{k-1}\tq{k}$,
  which together with \eqref{eq:A3} gives
  \begin{align*}
    A_3 &\geq \Bigl(2(k-1)(B_n-k) -(k-2)(B_n-k-1) - k(B_n-k+1)\Bigr)\tq{k-1}\tq{k} = -2\tq{k-1}\tq{k}.
  \end{align*}
  Hence
  \begin{align*}
    A_1+A_2+A_3 &\geq \tq{k-1}^2 + \tq{k}^2- 2\tq{k-1}\tq{k} = (\tq{k-1}-\tq{k})^2\geq 0,
  \end{align*}
  thus extending the induction.
\end{proof}

Next, we establish a variant of log-concavity with a similar but more complicated proof.
\begin{lemma}\thlabel{lem:lc.tech}
  For all $n\geq 0$ and all $k$,
  \begin{align}\label{eq:lc.tech}
    (B_n  - k)\bigl(\tq[n]{k}\bigr)^2 - (B_n-k-1) \tq[n]{k-1}\tq[n]{k+1} - \tq[n]{k-1}\tq[n]{k} &\geq 0.
  \end{align}
\end{lemma}
\begin{proof}
   We proceed by induction. 
   Let $\Eq[n]{k}$ be the left-hand side of \eqref{eq:lc.tech}.
   Since $\tq[0]{k} = \1\{k=w\}$ and $B_0=w+b$, we have
   $\Eq[0]{k} = b\1\{k=w\}$, demonstrating \eqref{eq:lc.tech} 
   when $n=0$. Now we assume $\Eq[n]{k}\geq 0$ for all $k$, and we show $\Eq[n+1]{k}\geq 0$ 
   for all $k$. It suffices to prove $\Eq[n+1]{k}\geq 0$ under the assumption
   that $B_{n+1}=B_n+1$, because $B_{n+1}$ is at least this large, and we can
   see that $\Eq[n+1]{k}$ is increasing in $B_{n+1}$ by writing it as
   as
   \begin{align*}
     \Eq[n+1]{k} = (B_{n+1}  - k-1)\Bigl[\bigl(\tq[n+1]{k}\bigr)^2 - \tq[n+1]{k-1}\tq[n+1]{k+1}\Bigr]
        + \bigl(\tq[n+1]{k}\bigr)^2 - \tq[n+1]{k-1}\tq[n+1]{k}
   \end{align*}
   and applying \thref{lem:log.concave}.
   
   For the sake of readability, we write
   $\tq{k}$ for $\tq[n]{k}$ and $B$ for $B_n$ in this proof. 
   We apply \eqref{eq:recurrence} to obtain
   \begin{align*}
     \bigl(\tq[n+1]{k}\bigr)^2 &= \frac{1}{B^2}\Bigl(
         (k-1)^2\tq{k-1}^2 + 2(k-1)(B-k)\tq{k-1}\tq{k} + (B-k)^2\tq{k}^2\Bigr),\\
    \tq[n+1]{k-1}\tq[n+1]{k+1} 
                                   &= \frac{1}{B^2}\Bigl((k-2)k\tq{k-2}\tq{k} + (k-2)(B-k-1)\tq{k-2}\tq{k+1}\\
                               &\qquad\qquad+k(B-k+1)\tq{k-1}\tq{k} + (B-k-1)(B-k+1)\tq{k-1}\tq{k+1}\Bigr),\\
    \tq[n+1]{k-1}\tq[n+1]{k} &=\frac{1}{B^2}\Bigl((k-2)(k-1)\tq{k-2}\tq{k-1} + (k-2)(B-k)\tq{k-2}\tq{k}\\
    & \qquad\qquad+
       (k-1)(B-k+1)\tq{k-1}^2 + (B-k)(B-k+1)\tq{k-1}\tq{k}\Bigr).
   \end{align*}
   Now, under the assumption that $B_{n+1}=B+1$, we expand $\Eq[n+1]{k}$ as $(A_1+A_2+A_3)/B^2$, where
   \begin{align*}
     A_1 &= (k-2)\Bigl((B-k+1)(k-1) \tq{k-1}^2
                        -(k+1)(B-k)\tq{k-2}\tq{k}
                        -(k-1) \tq{k-2}\tq{k-1}\Bigr)\\
         &= (k-2)\Bigl((k-1)\Eq{k-1} - 2(B-k)\tq{k-2}\tq{k}\Bigr) \geq -2(k-2)(B-k)\tq{k-2}\tq{k},
   \end{align*}
   and
   \begin{align*}
     A_2 &= (B-k+1)(B-k)\Bigl((B-k)\tq{k}^2 - (B-k-1)\tq{k-1}\tq{k+1}\Bigr)\\
         &= (B-k+1)(B-k)\Bigl( \Eq{k}+\tq{k-1}\tq{k} \Bigr) \geq (B-k+1)(B-k)\tq{k-1}\tq{k},
   \end{align*}
   and
   \begin{align*}
     A_3 &= (B-k+1)(B-k)(k-3)t_{k-1}t_k -(k-2)(B-k)(B-k-1)t_{k-2}t_{k+1}\\
       &= (B-k)\Bigl( (k-2)\bigl[ (B-k+1)\tq{k-1}\tq{k} - (B-k-1)\tq{k-2}\tq{k+1}\bigr] -(B-k+1)\tq{k-1}\tq{k}\Bigr)\\
       &= (B-k)\biggl( (k-2)\biggl[\frac{\Eq{k}\tq{k-2} + \Eq{k-1}\tq{k}}{\tq{k-1}} + 2\tq{k-2}\tq{k}\biggr] -(B-k+1)\tq{k-1}\tq{k}\biggr)\\
       &\geq (B-k)\Bigl(2(k-2)\tq{k-2}\tq{k}+(B-k+1)\tq{k-1}\tq{k}\Bigr),
   \end{align*}
   where we have applied the inductive hypothesis in each final step.
   Combining these bounds,
   \begin{align*}
     \Eq[n+1]{k} &\geq \frac{(B-k)}{B^2}\biggl(
        -2(k-2)\tq{k-2}\tq{k}
        + (B-k+1)\tq{k-1}\tq{k}\\
        &\qquad\qquad\qquad\qquad+ 2(k-2)\tq{k-2}\tq{k}-(B-k+1)\tq{k-1}\tq{k}\biggr)\\
        &=0.\qedhere
   \end{align*}
\end{proof}

\begin{proof}[Proof of \thref{lem:lc.variant}]
  First, we dispense with the case that any of $\tq[n]{k-1}$, $\tq[n+1]{k+1}$,
  $\tq[n]{k}$, or $\tq[n+1]{k}$ are equal to zero.
  If either of $\tq[n]{k-1}$ or $\tq[n+1]{k+1}$
  equals zero, then the left-hand side of \eqref{eq:lc.variant} is zero and the inequality holds.
  If $\tq[n]{k}$ or $\tq[n+1]{k}$ equals $0$ then
  $\tq[n+1]{k+1}=0$, since the support of $\tq[n]{k}$ is $\{w,\ldots,w+n\}$;
  in this case both sides of \eqref{eq:lc.variant} are zero.
  Thus we assume from now on that these four terms are all nonzero.

  Now, proving the lemma is equivalent to showing
  $\tq[n+1]{k}/\tq[n+1]{k+1} - \tq[n]{k-1}/\tq[n]{k}\geq 0$.
  We compute
  \begin{align*}
    \frac{\tq[n+1]{k}}{\tq[n+1]{k+1}} - \frac{\tq[n]{k-1}}{\tq[n]{k}}
       &= \frac{\tq[n+1]{k} - \frac{\tq[n]{k-1}}{\tq[n]{k}} \tq[n+1]{k+1}}{\tq[n+1]{k+1}}\\
       &= \frac{1}{\tq[n+1]{k+1}}\Biggl( 
         \biggl(\frac{k-1}{B_n}\biggr) \tq[n]{k-1} + \biggl(1 - \frac{k}{B_n}\biggr)\tq[n]{k}\\
           &\qquad\qquad\qquad- \frac{\tq[n]{k-1}}{\tq[n]{k}}\biggl[\biggl(\frac{k}{B_n}\biggr) \tq[n]{k} + \biggl(1 - \frac{k+1}{B_n}\biggr)\tq[n]{k+1}\biggr]\Biggr)\\
       &= \frac{1}{B_n\tq[n+1]{k+1}\tq[n]{k}}\Biggl( -\tq[n]{k-1}\tq[n]{k} + (B_n-k)\bigl(\tq[n]{k}\bigr)^2  - (B_n -k-1)\tq[n]{k-1}\tq[n]{k+1}\Biggr),
  \end{align*}
  which is nonnegative by \thref{lem:lc.tech}.  
\end{proof}

\begin{remark}\thlabel{rmk:p<1.use}
  It is possible to avoid all the work of this appendix, at the cost of a slightly inferior
  concentration bound for $N_n(1,w)$. The result of this appendix (\thref{lem:lc.variant})
  is used to prove that $N_{n-l}^{[l+1]}(b,w) + l \succeq N_n^{(l+1)}(b,w)$ (\thref{lem:unfactorialize}),
  which is then applied in the proof of Proposition~\ref{ineq}. An alternate path is to invoke the 
  following stochastic inequality between the factorial and power bias transformations,
  which holds for any nonnegative random variable:
  \begin{align}\label{eq:power.to.fac}
    X^{(l+1)}\utailprec{p} X^{[l+1]},
  \end{align}
  where
  \begin{align*}
    p = \frac{\E X^{l+1}}{\E\bigl[ X (X+1)\cdots (X+l) \bigr]}.
  \end{align*}
  Modifying the derivation in Proposition~\ref{ineq} slightly, we get
  \begin{align*}
    N_n(1,w) &\eqd Q_w (N_{n-l}(1,w+1+l)-w-1) \\
      &\succeq Q_w(N_n(1,w+1+l) - l -w -1,
  \end{align*}
  with the second line holding since at most $l$ white balls
  can be added from steps $n-l$ to $n$. Then following the same
  steps as in Proposition~\ref{ineq},
  \begin{align*}
    N_n(1,w) &\succeq
        Q_w(N_n^{[l+1]}(1,w)-w) \succeq V_wN_n^{[l+1]}(1,w).
  \end{align*}
  Finally, invoking \eqref{eq:power.to.fac}, we have 
  \begin{align*}
    N_n(1,w)\utailsucc{p} V_wN_n^{(l+1)}(1,w)\eqd N_n^*(1,w).
  \end{align*}
  The concentration bounds obtained from this are worse because of the factor of $p$ in the exponent,
  but it does illustrate how the $p<1$ versions of our concentration bounds can be used.
\end{remark}

\section*{Appendix~B}
  In this appendix, we prove \thref{prop:reliability} for the convenience of the reader.
  See also \cite[Chapters~4 and 6]{BP} for more background material.

  \begin{proof}[Proof of \thref{prop:reliability}\ref{i:ifr.char}]
    Suppose that $X$ is D-IFR and let $p_n = \P[X=n\mid X\geq n]$.
    To show that $[X-k\mid X>k]$ is stochastically decreasing in $k$,
    construct a random variable $T$
    by the following procedure: Fix some $k\geq 0$. Start at 1 and halt with probability $p_{k+1}$;
    otherwise advance to 2 and halt with probability $p_{k+2}$; otherwise advance to $3$,
    and continue like this, letting $T$ be the value where we halt. It is evident
    that $T\sim[X-k\mid X>k]$. Since $p_n$ is increasing, we are more likely to halt
    at each step when $k$ is increased. By a simple coupling, this demonstrates
    that $[X-k\mid X>k]$ is stochastically decreasing in $k$.
    
    Conversely, suppose that $[X-k\mid X > k]$ is stochastically decreasing in $k$.
    Then
    \begin{align*}
      \P\bigl[X-k+1\leq 1 \bigmid X>k-1\bigr] \leq \P\bigl[X-k\leq 1\bigmid X>k\bigr]
    \end{align*}
    by the definition of stochastic dominance, which proves that
    \begin{align*}
      \P[X=k\mid X\geq k]&\leq \P[X=k+1\mid X\geq k+1].\qedhere
    \end{align*}
  \end{proof}
  
  \begin{proof}[Proof of \thref{prop:reliability}\ref{i:nbue.char}]
    From the definition of the discrete equilibrium transform,
    \begin{align*}
      \P[X^e=n] = \sum_{k=n}^{\infty} \frac{1}{k}\P[X^s=k] = \sum_{k=n}^{\infty} \frac{1}{\E X}\P[X=k]
        = \frac{1}{\E X} \P[X\geq n].
    \end{align*}
    Hence
    \begin{align}
      \P[X^e > k] &= \frac{1}{\E X}\sum_{n=k+1}^\infty\P[X\geq n]\label{eq:Xe>k}\\
      &=\frac{1}{\E X}\E\bigl[ (X-k)\1\{X> k\}\bigr]
      = \P[X> k]\frac{\E[X-k\mid X> k]}{\E X}.\nonumber
    \end{align}
    Therefore $\P[X^e > k]\leq \P[X>k]$ holds for all $k\geq 1$ if and only if
    $\E[X-k\mid X>k]\leq \E X$ for all $k\geq 1$.
  \end{proof}

  \begin{proof}[Proof of \thref{prop:reliability}\ref{i:hierarchy}]
    Suppose $X$ takes values in the positive integers and is log-concave.
    Let $p_n=\P[X=n]$, and let $N$ be the highest value such that $p_N>0$,
    with $N=\infty$ a possibility. From the definition of log-concave,
    \begin{align*}
      \frac{p_{n-1}}{p_n}\leq \frac{p_n}{p_{n+1}}
    \end{align*}
    for all $1\leq n\leq N$. This implies that for any fixed $k$, the ratio
    \begin{align*}
      \frac{\P[X - k =n \mid X>k]}{ \P[X-k-1=n\mid X>k+1]} = \frac{p_{n+k}}{p_{n+k+1}}\cdot \frac{\P[X>k+1]}{\P[X>k]}
    \end{align*}
    is increasing in $n$, and this condition implies
    that $[X-k\mid X>k]$ stochastically dominates $[X-k-1\mid X>k+1]$
    (see \cite[Theorem~1.C.1]{SS}).
    Hence $X$ is D-IFR by \thref{prop:reliability}\ref{i:ifr.char}.
    
    Now, suppose that $X$ is D-IFR. It follows from \thref{prop:reliability}\ref{i:ifr.char}
    that $\E[X-k\mid X>k]$ is decreasing in $k$ for integers $k\geq 0$, proving that
    \begin{align*}
      \E[X-k\mid X>k] \leq \E[X-0\mid X>0] = \E X.
    \end{align*}
    By \thref{prop:reliability}\ref{i:nbue.char}, this shows that $X$ is NBUE.
  \end{proof}

 \bigskip
\section*{Acknowledgments} 

\noindent The initial portion of this work was conducted at the meeting \emph{Stein's method and applications in high-dimensional statistics} held at the American Institute of Mathematics in August 2018. We would also like to express our gratitude to John Fry and staff Estelle Basor, Brian Conrey, and Harpreet Kaur at the American Institute of Mathematics for  the generosity and excellent hospitality in hosting this meeting at the Fry's Electronics corporate headquarters in San Jose, CA, and Jay Bartroff, Larry Goldstein, Stanislav Minsker and Gesine Reinert for organizing such a stimulating meeting. 
T.J.\ received support from NSF grant DMS-1811952 and PSC-CUNY Award \#62628-00 50.

 \bibliographystyle{amsalpha}
\bibliography{stein_concentration}

\end{document}